\documentclass{article}
\usepackage{amsmath, amsthm, amssymb, bm, mathtools}
\usepackage[reset, a4paper, vmargin=3truecm, hmargin=2truecm]{geometry}
\usepackage{xcolor}

\usepackage[shortlabels]{enumitem}
\setlist[enumerate]{label=\((\arabic*)\),itemjoin=\qquad} 
\usepackage[nobreak]{cite}
\usepackage{mleftright}
\mleftright

\DeclareMathAlphabet\matheuscr{U}{eus}{m}{n}

\theoremstyle{plain}
\newtheorem{thm}{Theorem}[section]
\newtheorem{lem}[thm]{Lemma}
\newtheorem{prop}[thm]{Proposition}
\newtheorem{cor}[thm]{Corollary}
\newtheorem{conject}[thm]{Conjecture}
\newtheorem{prob}[thm]{Problem}

\theoremstyle{definition}
\newtheorem{dfn}[thm]{Definition}
\newtheorem{ex}{Example}

\theoremstyle{remark}
\newtheorem{rem}{Remark}[section]

\newcommand{\Z}{\mathbb{Z}}
\newcommand{\Q}{\mathbb{Q}}
\newcommand{\R}{\mathbb{R}}
\newcommand{\C}{\mathbb{C}}
\newcommand{\e}{\varepsilon}
\newcommand{\liesl}{\mathfrak{sl}}

\newcommand{\Zp}{\Z_p}
\newcommand{\Qp}{\Q_p}
\newcommand{\Cp}{\C_p}
\newcommand{\CZp}{{\mathrm C}\Z_p}
\newcommand{\cQp}{\overline{\Q}_p}

\DeclareMathOperator{\vol}{vol}
\DeclareMathOperator{\Tr}{Tr}
\DeclareMathOperator{\tr}{tr}
\DeclareMathOperator*{\Res}{Res} 

\newcommand{\apery}[1]{A_{#1}}
\newcommand{\bpery}[1]{B_{#1}}
\newcommand{\Apery}[1]{\matheuscr{A}_{#1}}
\newcommand{\Bpery}[1]{\matheuscr{B}_{#1}}
\newcommand{\Rpery}[1]{\matheuscr{R}_{#1}}

\newcommand{\HRabi}{H_{\mathrm{QRM}}}
\newcommand{\PRabi}{Z_{\mathrm{QRM}}}

\newcommand{\w}{t}

\newcommand{\tZ}{\widetilde{Z}}
\newcommand{\tB}{\widetilde{B}}
\newcommand{\vu}{\bm{u}}
\newcommand{\J}[2]{J_{#1}(#2)}
\newcommand{\tJ}[2]{{\widetilde J}_{#1}(#2)}
\newcommand{\evenZ}{Z^{\text{\upshape even}}}
\newcommand{\oddZ}{Z^{\text{\upshape odd}}}
\newcommand{\dG}[2][{}]{\mathrm{dG}_{#2}^{#1}} 
\newcommand{\hecke}{\mathfrak{G}(2)}

\newcommand{\set}[3][undefined]{\expandafter\ifx\csname#1\endcsname\undefined \left\{#2\,\middle|\,#3\right\}%
\else\csname#1l\endcsname\{#2\,\csname#1\endcsname|\,#3\csname#1r\endcsname\}\fi}
\newcommand{\mat}[1]{\begin{pmatrix}#1\end{pmatrix}}
\newcommand{\abs}[1]{\left\lvert#1\right\rvert} 
\newcommand{\pabs}[1]{\abs{#1}_p} 
\newcommand{\deq}{\coloneqq} 
\newcommand{\psum}{\sideset{}{'}{\sum}} 
\newcommand{\hgf}[5]{{}_{#1}F_{#2}\left(#3;#4;#5\right)} 

\newcommand{\MSC}[2]{\medskip\noindent
\textbf{2010 Mathematics Subject Classification:} \textit{Primary} #1, \textit{Secondary} #2.}
\newcommand{\KeyWords}[1]{\medskip\noindent{\bfseries Keywords and phrases:} {#1}}
\newcommand{\CREST}{JST CREST Grant Number JPMJCR2113, Japan}
\newcommand{\dedicatory}[1]{{\centering\itshape #1\par}\bigskip}

\title{Partition functions for non-commutative harmonic oscillators\\ and related divergent series}

\author{Kazufumi Kimoto%
\thanks{Partially supported by Grant-in-Aid for Scientific Research (C) No.22K03272, JSPS, Japan.}
\, and \,
Masato Wakayama%
\thanks{Partially supported by \CREST.}
}

\begin{document}

\maketitle

\dedicatory{Dedicated to the late Gerrit van Dijk and the memory of our long friendship.}

\begin{abstract}
In the standard normalization, the eigenvalues of the quantum harmonic oscillator are given by positive half-integers with the Hermite functions as eigenfunctions. Thus, its spectral zeta function is essentially given by the Riemann zeta function. The heat kernel (or propagator) of the quantum harmonic oscillator (qHO) is given by the Mehler formula, and the partition function is obtained by taking its trace. In general, the spectral zeta function of the given system is obtained by the Mellin transform of its partition function. In the case of non-commutative harmonic oscillators (NCHO), however, the heat kernel and partition functions are still unknown, although meromorphic continuation of the corresponding spectral zeta function and special values at positive integer points have been studied. On the other hand, explicit formulas for the heat kernel and partition function have been obtained for the quantum Rabi model (QRM), which is the simplest and most fundamental model for light and matter interaction in addition to having the NCHO as a covering model. In this paper, we propose a notion of the \emph{quasi-partition function} for a quantum interaction model if the corresponding spectral zeta function can be meromorphically continued to the whole complex plane. The quasi-partition function for qHO and QRM actually gives the partition function.
Assuming that this holds for the NCHO (currently a conjecture), we can find various interesting properties for the spectrum of the NCHO. Moreover, although we can not expect any functional equation of the spectral zeta function for the quantum interaction models, we try to seek if there is some relation between the special values at positive and negative points.  
Attempting to seek this, we encounter certain divergent series expressing formally the Hurwitz zeta function by calculating integrals of the partition functions. We then give two interpretations of these divergent series by the Borel summation and $p$-adically convergent series defined by the $p$-adic Hurwitz zeta function.

\MSC{11M41}{11A07, 11F03, 33C20, 81V80}

\end{abstract}

\tableofcontents

\section{Introduction}
\label{sec:introduction}

The non-commutative harmonic oscillator (NCHO) was introduced in \cite{PW2001, PW2002} as a non-commutative generalization of the usual quantum harmonic oscillator (qHO) (see also \cite{P2010, P2014}). Actually, in addition to the canonical commutation relation $[\frac{d}{d x}, x]=1$ (CCR) for the qHO, there is another non-commutativity by the two (constant) matrices appearing in the Hamiltonian.
Particularly, when these matrices commute, the NCHO is found to be unitarily equivalent to a couple of the qHOs. The NCHO has been studied from diverse points of view, including the description of the spectrum using semi-classical analysis \cite{P2008}, the study of the multiplicity of the lowest eigenvalue \cite{HS2013,HS2014}, applications to PDE via generalizations of inequalities such as the Fefferman-Phong inequality (see \cite{P2014} and the references therein), and the arithmetics in the spectrum via the associated spectral zeta function \cite{KW2023}. In particular, the study of the special values of the spectral zeta function of the NCHO has revealed a rich arithmetic theory including elliptic curves, modular forms and natural extensions of Eichler forms (automorphic integrals) with associated cohomology group. It is actually an extended study of the special values of the Riemann zeta function $\zeta(s)$, particularly of the irrationality proof of $\zeta(3)$ initiated by Roger Apéry and the deeper study from the point of view of algebraic geometry and modular forms by F. Beukers (See \cite{Beu1983, Beu1985, Beu1987, BP1984, KW2023} and the references therein). Of course, the first requirement for these studies was the analytic continuation of the spectral zeta function to the whole complex plane. However, it seems difficult to obtain explicit formulas for the heat kernel and partition function of the NCHO  (see, Section \ref{sec:NCHO} for the definition of the partition function) and it has not succeeded yet, so the analytic continuation as a meromorphic function was performed by the parametrix method \cite{IW2005a, IW2005KJM} (i.e., the asymptotic expansion for $t\downarrow 0$ of the heat operator). 

In this paper, in addition to the qHO, we consider spin-boson models such as the quantum Rabi model (QRM) and asymmetric quantum Rabi model (AQRM) as a reference point of the study. Here, the QRM (and AQRM) describes the fundamental interaction between a photon and a two-level atom \cite{BCBS2016JPA}. In fact, for these models, the respective heat kernel, whence the partition function has been obtained explicitly \cite{RW2021, R2023}. From that the meromorphic continuations of their spectral zeta functions (of Hurwitz's type) are obtained by the contour integral expressions as in the case of the Riemann zeta function (see e.g. \cite{T1987}). The Weyl law for the spectrum follows, too. 
 
Although the spectrum of the NCHO can not be described explicitly, employing the oscillator representation of $\liesl_2$ and constructing a certain Laplace type intertwining operator, we have a Fuchsian ODE (Heun ODE) picture for the eigenvalue problem of the Hamiltonian \cite{O2001CMP, W2016}. Using this picture, we observe that there is an interesting connection between the NCHO and QRM via their respective Heun ODE pictures \cite{W2016}. In fact, the eigenvalue problem of the NCHO (resp. QRM) is equivalent to the existence of the holomorphic solution of the Heun ODE \cite{SL2000} on some complex domain that includes $0, 1$ but excludes $\alpha\beta(>1)$, which has four regular singular points at $0, 1, \alpha\beta, \infty$ (resp. entire solution of the confluent Heun ODE which has two regular points at $0, 1$ and irregular singularity at $\infty$). In this situation, some appropriate confluent process gives the latter confluent Heun ODE from the former. In this sense, the NCHO can be regarded as a covering model of the QRM. The covering relation of these pictures is generalized to the case between AQRM and $\eta$-NCHO \cite{RW2023CMP}. 

The first point of this study is to define a ``quasi-partition function'' as a candidate for the partition function for the NCHO (and for a general Hamiltonian system)  in the reverse direction, that is, starting with the special values at non-positive integers of the spectral zeta function which was obtained in the formula for its meromorphic continuation in \cite{IW2005a}. 
We expect that the quasi-partition function gives the partition function of the NCHO, because

\begin{enumerate}
\item in the cases of the qHO and QRM (and AQRM), this is indeed the case \cite{RW2019, RW2021, R2023}.
\item the NCHO is a matrix-variant of the qHO and is seen as a covering model of the QRM \cite{W2016} (see also \cite{RW2023CMP} by the Heun ODE picture of the respective eigenvalue problem).
\end{enumerate}

The second purpose is to study certain divergent series appearing in the evaluation of the partition function, which to the best knowledge of the authors have not been studied before. Actually, in the case of the qHO, such divergent series appear as the special values at $s=2$ of the Hurwitz zeta function, and actually are shown to converges $p$-adically in Cohen's book\cite{HC2007} but with a different motivation.
We investigate a broader range of special values for the NCHO.
The motivation of dealing these divergent series is coming up the absence of any functional equation for the spectral zeta function of the NCHO. Actually, we have the question as ``is there any relationship between special values at positive integers and those at non-positive integers of the spectral zeta function?''



\begin{figure}[htbp]\label{fig.diveregent}
\centering
\includegraphics{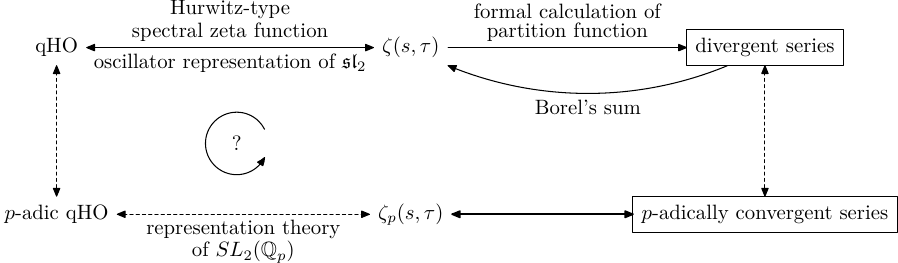}
\caption{Divergent series expressing $\zeta(s,\tau)$}
\end{figure}

The paper is organized as follows.
In Section \ref{sec:NCHO}, we briefly recall the basic facts on the NCHO including the meromorphic continuation of its spectral zeta function \cite{IW2005a}.
In Section \ref{sec:SV_RZ}, we recall the known-results of the special values of the Riemann zeta function at the integral points in relation with a similar study for the NCHO
(The contents of Subsections \ref {sec:Geometry-RZ} and \ref {sec:Geometry-NCHO-Zeta} demonstrate a major motivation for conducting this research, but the readers can understand the subsequent chapters even if they are skipped.
We provide a comparative description between Apéry and Apéry-like numbers in Appendix \ref{Appendix_Apery},
which would be a complement to these subsections).
We introduce the comparison between these for the motivation presented in this paper. 
We present the necessary facts on the QRM and AQRM from the view points of the partition functions and their spectral zeta functions (of the Hurwitz type) in Section \ref{sec:QRM}. 
We introduce the notion of quasi-partition function using the special values at negative integer points of the spectral zeta function in Section \ref{sec:quasiPF}. Then, we derive some interesting results if the quasi-partition function of the NCHO is identified to its partition function.
In Section \ref{sec:DivergentSeries}, we examine the evaluation of the primitive type of divergent series
that arise from the Hurwitz zeta function, using the very definition of the partition function of the qHO.
In Section \ref{sec:Interpretation} we discuss certain interpretation of the divergent series from the view points of the Borel sum and $p$-adic Hurwitz zeta function. 
In the final section, we revisit the discussion about special values of the spectral zeta function for NCHO and give a problem. 


\begin{rem}
The original Rabi model, which is a semiclassical version of the QRM, was introduced by Isidor Issac Rabi in 1936 and later in 1965 fully quantized in \cite{JC1963}. For more than a decade, particularly after the proof of unexpected integrability of the QRM \cite{B2011PRL}, there have been extensive studies, both theoretical and experimental (see e.g., \cite{BCBS2016JPA, YS2018}), on the QRM and its associated models, and their generalization aiming towards applications to condensed matter, quantum information science and engineering (see e.g., \cite{HR2008}). 
\end{rem}

\section{Non-commutative harmonic oscillator (NCHO)}
\label{sec:NCHO}

Let $Q$ be a parity preserving matrix valued ordinary differential operator defined by
\begin{equation*}
Q=Q_{\alpha,\beta}
:=\begin{pmatrix}\alpha & 0 \\ 0 & \beta\end{pmatrix}\Bigl(-\frac12\frac{d^2}{dx^2}+\frac12x^2\Bigr)
+\begin{pmatrix}0 & -1 \\ 1 & 0\end{pmatrix}\Bigl(x\frac{d}{dx}+\frac12\Bigr).
\end{equation*}
The system defined by $Q$ is called the \emph{non-commutative harmonic oscillator} (NCHO). 
In this paper,
we always assume that $\alpha,\beta>0$ and $\alpha\beta>1$.
Under this assumption, the operator $Q$ becomes a positive self-adjoint unbounded operator on $L^2(\R;\C^2)$,
the space of $\C^2$-valued square-integrable functions on $\R$,
and $Q$ has only a discrete spectrum with uniformly bounded multiplicity:
\begin{equation*}
0<\lambda_1\le\lambda_2\le\lambda_3\le\dots(\nearrow\infty).
\end{equation*}
It was proved that the lowest eigenstate is multiplicity free
and the multiplicity of general eigenstate is less than or equal to $2$.

The spectral zeta function $\zeta_Q(s)$ of $Q$ \cite{IW2005a} is defined by
\begin{equation*}
\zeta_Q(s)\deq \sum_{n=1}^\infty \lambda_n^{-s} \quad (\Re(s)>1).
\end{equation*}
It is noted that, when $\alpha=\beta$,
$Q=Q_{\alpha,\alpha}$ is unitarily equivalent to a couple of the quantum harmonic oscillators,
whence the eigenvalues are easily calculated as
$\set{\sqrt{\alpha^2-1}\bigl(n+\frac12\bigr)}{n\in \Z_{\geq0}}$
having multiplicity $2$.
Namely, in this case, $\zeta_Q(s)$ is essentially given by the Riemann zeta function $\zeta(s)$
as $\zeta_Q(s)=2(2^s-1)(\alpha^2-1)^{-s/2}\zeta(s)$.
In other words, $\zeta_Q(s)$ is a $\frac{\alpha}{\beta}$-analog of $\zeta(s)$.
The clarification of the spectrum in the general $\alpha\not=\beta$ case is, however,
considered to be highly non-trivial.
Indeed, while the spectrum is described theoretically by using certain continued fractions
and also by Heun's ordinary differential equations (those have four regular singular points)
in a certain complex domain, almost no satisfactory information on each eigenvalue is available in reality
when $\alpha\ne\beta$.

The heat kernel $K_H(x,y,t)$ of a quantum system defined by the Hamiltonian $H$ is the integral kernel
corresponding to the heat operator \(e^{-t H} \), that is, 
\[
e^{-t H}\phi(x)= \int_{-\infty}^\infty  K_H(x,y,t) \phi(y)dy
\]
for a compactly supported smooth (vector valued) function $\phi$ on $\R$.
Precisely, $K_H(x,y,t)$ is the (e.g., two-by-two matrix valued when $H$ defines the NCHO) function
satisfying $\frac{\partial}{\partial t} K_H(x,y,t)= -H K_H(x,y,t)$
for all $t>0$ and $\lim_{t \to 0}K_H(x,y,t)=\delta_x(y)\bm{I}_2$ for $x,y \in \R$. 
Throughout the paper, we assume that $H$ has only discrete spectrum
and the dimensions of eigenstates are uniformly bounded. 
Then, the partition function $Z_H(\beta)$ of $H$ is given by the trace of the Boltzmann factor $e^{-\beta E(\mu)}$,
$E(\mu)$ being the energy of the state $\mu$:
\[
Z_H(\beta) \deq \text{Tr}[e^{-t H}]= \sum_{\mu\in \Omega} e^{-\beta E(\mu)},
\]
where $\Omega$ denotes the set of all possible eigenstates of $H$. 

Now we denote by $Z_Q(t)$ the partition function of the NCHO:
\begin{align}
Z_Q(t)=\Tr K_Q(t)=\int \tr K_Q(x,x,t)dx= \sum_{j=1}^\infty e^{-\lambda_jt} \quad (t>0),
\end{align}
where $K_Q(x,y,t)$ is the heat kernel of the NCHO.
Then, $\zeta_Q(s)$ is given by the Mellin transform of $Z_Q(t)$ as
\[
\zeta_Q(s)= \frac1{\Gamma(s)}\int_0^\infty t^{s-1}Z_Q(t)dt \quad (\Re(s)> 1). 
\]
Although no explicit formula of $K_Q(x,y,t)$ and $Z_Q(t)$ have been obtained,
it is known that $\zeta_Q(s)$ can be meromorphically extended to the whole complex plane.
Actually, we recall the following formula which gives an analytic continuation of $\zeta_Q(s)$ from \cite{IW2005a}. 

\begin{thm}[Main Theorem of \cite{IW2005a}] \label{MT_NCHO}
There exist constants $C_{Q,j}\, (j=1,2,\ldots)$ such that for every positive integer $N$ we have 
\[
\zeta_Q(s)=\frac1{\Gamma(s)}\Bigg[\frac{\alpha+\beta}{\sqrt{\alpha\beta(\alpha\beta-1)}} \frac1{s-1}
+ \sum_{j=1}^N\frac{C_{Q,j}}{s+2j-1} +H_{Q, N}(s)\Bigg],
\]
where $H_{Q, N}(s)$ is a holomorphic function in $\Re(s)>-2N$.
Consequently, $\zeta_Q(s)$ is meromorphic in the whole complex plane having a unique simple pole at $s=1$
with $\Res_{s=1}\zeta_Q(s)=\frac{\alpha+\beta}{\sqrt{\alpha\beta(\alpha\beta-1)}}$
and has zeros for $s$ being non-positive even integers.
Moreover, we have 
\begin{equation}\label{Q-negative-values1}
\zeta_Q(1-2m)=C_{Q,m}(2m-1)! \quad (m=1,2,\ldots).
\end{equation}
\end{thm}

Based on this theorem, we introduce the quasi-partition function of the NCHO in Section \ref{sec:quasiPF}
by employing the special values at the negative integers. 

\begin{rem}
The proof of Theorem \ref{MT_NCHO} was obtained from the following asymptotic expansion
of the trace of the heat kernel of the NCHO for $t\downarrow 0$ \cite{IW2005a} by the parametrix method:
\[
Z_Q(t)=\Tr K_Q(t) \sim \frac{\alpha+\beta}{\sqrt{\alpha\beta(\alpha\beta-1)}}t^{-1}
+ C_{Q,1}t+ C_{Q,2}t^3+C_{Q,3}t^5+\cdots, \qquad 0<t \ll1. 
\]
\end{rem}

\section{Special values of $\zeta(s)$ and $\zeta_Q(s)$}
\label{sec:SV_RZ}

In this section, we first recall well-known facts on Bernoulli numbers
and special values of the Riemann zeta function $\zeta(s)$,
and similar properties of the special values of $\zeta_Q(s)$ for the NCHO
from number theoretic and geometric structures behind them.
We put the relevant congruence properties of Apéry numbers and Apéry-like numbers for the NCHO
in the Appendix \ref{Appendix_Apery}.

\subsection{Bernoulli numbers and polynomials}
\label{sec:Bernoulli}

We briefly review basic facts about the special values of the Riemann zeta function $\zeta(s)$ ($\Re s>1$): 
\begin{equation*}
\zeta(s):=\sum_{n=1}^\infty \frac1{n^s} \, 
\Big(=\prod_{p:\text{prime}}\frac1{1-p^{-s}}\;\text{:\; Euler's product}\Big).
\end{equation*}
This is analytically continued to the whole complex plane $\C$ as a single-valued meromorphic function
which has a unique simple pole at $s=1$. 

We denote by $B_k$ and $B_k(x)$ the Bernoulli numbers and Bernoulli polynomials defined by
\begin{equation*}
\sum_{k=0}^\infty B_k\frac{t^k}{k!}=\frac{t}{e^t-1},\qquad
\sum_{k=0}^\infty B_k(x)\frac{t^k}{k!}=\frac{te^{tx}}{e^t-1}.
\end{equation*}
Since $\frac{-t}{e^{-t}-1}=\frac{t}{e^t-1}+t$,
it follows that $B_0=1$, $B_1=-\frac12$ and $B_{2k+1}=0$ for $k\in\Z_{>0}$.
We write down first several Bernoulli numbers:
\begin{equation*}
B_2=\frac16, \quad
B_4=-\frac1{30}, \quad
B_6=\frac1{42}, \quad
B_8=-\frac1{30}, \quad
\ldots.
\end{equation*}
The special values of $\zeta(s)$ at positive even integers are
\begin{equation}\label{SB-Even}
\zeta(2k)=\sum_{n=1}^\infty\frac1{n^{2k}}=\frac{(-1)^{k-1}2^{2k-1}B_{2k}}{(2k)!}\pi^{2k}.
\end{equation}
For instance, we see that
\begin{equation*}
\zeta(2)=\frac{\pi^2}6 \;\text{(Basel's problem)},\qquad
\zeta(4)=\frac{\pi^4}{90},\qquad
\zeta(6)=\frac{\pi^6}{945}.
\end{equation*}

The Riemann zeta function $\zeta(s)$ is written in the form
\begin{equation*}
2\pi^{-s/2}\Gamma(s/2)\zeta(s)
=\int_0^\infty(\vartheta(u)-1)u^{s/2-1}du,
\end{equation*}
where
\begin{equation*}
\vartheta(u):=\sum_{n=-\infty}^\infty e^{-\pi n^2u} \quad (u>0)
\end{equation*}
is the Jacobi theta function.
The \emph{functional equation} for $\zeta(s)$ follows from $\vartheta(1/u)=\sqrt u\,\vartheta(u)$ (Poisson's summation formula):
\begin{equation}\label{functional equation of Riemann zeta function}
\xi(1-s)=\xi(s),
\end{equation}
where $\xi(s):=\pi^{-s/2}\Gamma(s/2)\zeta(s)$.
By the functional equation, the values of $\zeta(s)$ at negative integers are given by
\begin{equation*}
\zeta(-2k)=0,\qquad
\zeta(1-2k)=-\frac{B_{2k}}{2k} \quad (k=1,2,\dots).
\end{equation*}

\subsection{Geometry and arithmetics for $\zeta(n)$}
\label{sec:Geometry-RZ}

The values of $\zeta(s)$ at positive odd integers are rather mysterious.
Roger Apéry proved the irrationality of $\zeta(2)$ and $\zeta(3)$
by utilizing certain rational numbers, which are now called \emph{Apéry numbers}.
As for the irrationality of other odd special values, with the work of T.~Rivoal, Zudilin \cite{Zu2001} proved that at least one of the four values
$\zeta(5), \zeta(7), \zeta(9), \zeta(11)$ is irrational.
This is the best result on the irrationality of odd zeta values so far.

\emph{Apéry numbers for $\zeta(2)$} are given by
\begin{align*}
\apery2(n)&=\sum_{k=0}^n\binom nk^{\!\!2}\!\binom{n+k}k,\\
\bpery2(n)&=\sum_{k=0}^n\binom nk^{\!\!2}\!\binom{n+k}k
\biggl(2\sum_{m=1}^n\frac{(-1)^{m-1}}{m^2}+\sum_{m=1}^k\frac{(-1)^{n+m-1}}{m^2\binom nm\binom{n+m}m}\biggr).
\end{align*}
These numbers satisfy a recurrence relation of the same form
\begin{align}\label{eq:reccurence_of_A2}
n^2u(n)-(11n^2-11n+3)u(n-1)-(n-1)^2u(n-2)=0\quad(n\ge2)
\end{align}
with initial conditions $\apery2(0)=1, \apery2(1)=3$ and $\bpery2(0)=0, \bpery2(1)=5$.
The ratio $\bpery2(n)/\apery2(n)$ converges to $\zeta(2)$,
and this convergence is rapid enough to prove the irrationality of $\zeta(2)$.
Consider the generating functions
\begin{align*}
\Apery2(t)=\sum_{n=0}^\infty \apery2(n)t^n,\quad
\Bpery2(t)=\sum_{n=0}^\infty \bpery2(n)t^n,\quad
\Rpery2(t)=\Apery2(t)\zeta(2)-\Bpery2(t).
\end{align*}
It is proved that
\begin{align*}
L_2\Apery2(t)=0,\quad L_2\Bpery2(t)=-5,\quad L_2\Rpery2(t)=5,
\end{align*}
where $L_2$ is a differential operator given by
\begin{align*}
L_2=t(t^2+11t-1)\frac{d^2}{dt^2}+(3t^2+22t-1)\frac{d}{dt}+(t+3).
\end{align*}
The function $\Rpery2(t)$ is also expressed as follows:
\begin{align*}
\Rpery2(t)=\int_0^1\!\int_0^1\frac{dxdy}{1-xy+txy(1-x)(1-y)}.
\end{align*}
The family $Q^2_t\colon 1-xy+txy(1-x)(1-y)=0$ of algebraic curves,
which comes from the denominator of the integrand,
is birationally equivalent to the universal family $C^2_t$ of elliptic curves having rational $5$-torsion.
Moreover, the differential equation $L_2\Apery2(t)=0$ is regarded as a Picard-Fuchs equation for this family,
and $\Apery2(t)$ is interpreted as a period of $C^2_t$.
If we take $t$ to be an appropriate modular function with respect to $\Gamma_1(5)$,
then $\Apery2(t)$ becomes a modular form of weight $1$ (see the pioneering insight in \cite{Beu1983}).

Similarly, the \emph{Apéry numbers} $\apery3(n)$ and $\bpery3(n)$ for $\zeta(3)$ are defined as binomial sums
in a manner similar to $\apery2(n)$ and $\bpery2(n)$,
and satisfy a recurrence relation of the same form
\begin{equation}\label{eq:recurrence_of_A3}
n^3u(n)-(34n^3-51n^2+27n-5)u(n-1)+(n-1)^3u(n-2)=0\quad(n\ge2)
\end{equation}
with initial conditions $\apery3(0)=1, \apery3(1)=5$ and $\bpery3(0)=0, \bpery3(1)=6$.
The ratio $\bpery3(n)/\apery3(n)$ converges to $\zeta(3)$ rapidly enough
to allow us to prove the irrationality of $\zeta(3)$.
The generating functions
\begin{align*}
\Apery3(t)=\sum_{n=0}^\infty \apery3(n)t^n,\quad
\Bpery3(t)=\sum_{n=0}^\infty \bpery3(n)t^n,\quad
\Rpery3(t)=\Apery3(t)\zeta(3)-\Bpery3(t).
\end{align*}
satisfy the differential equation
\begin{align*}
L_3\Apery3(t)=0,\quad L_3\Bpery3(t)=5,\quad L_3\Rpery3(t)=-5,
\end{align*}
where $L_3$ is a certain differential operator,
and the function $\Rpery3(t)$ is expressed in the form
\begin{align*}
\Rpery3(t)=\int_0^1\!\int_0^1\!\int_0^1
\frac{dxdydz}{1-(1-xy)z-txyz(1-x)(1-y)(1-z)}.
\end{align*}
The family $Q^3_t\colon 1-(1-xy)z-txyz(1-x)(1-y)(1-z)=0$ of algebraic surfaces
coming from the denominator of the integrand is birationally equivalent to
a certain family $C^3_t$ of $K3$ surfaces with Picard number $19$.
Furthermore, the differential equation $L_3\Apery3(t)=0$ is regarded as a Picard-Fuchs equation for this family,
and $\Apery3(t)$ is interpreted as a period of $C^3_t$.
If we take $t$ to be an appropriate modular function with respect to $\Gamma_1(6)$,
then $\Apery3(t)$ becomes a modular form of weight $2$ (see \cite{BP1984, Beu1987}).

\subsection{Geometry and arithmetics for $\zeta_Q(n)$}
\label{sec:Geometry-NCHO-Zeta}

The special values of $\zeta_Q(s)$ are given by the formula
\begin{equation}\label{eq:specialvalues of zeta_Q(s)}
\zeta_Q(k)
=2\biggl(\frac{\alpha+\beta}{2\sqrt{\alpha\beta(\alpha\beta-1)}}\biggr)^{\!\!k}
\Biggl(
\zeta(k,1/2)+\sum_{0<2j\le k}\biggl(\frac{\alpha-\beta}{\alpha+\beta}\biggr)^{\!\!2j}R_{k,j}(\kappa)
\Biggr)
\end{equation}
for $k\ge2$ (Theorem 2.6 in \cite{KW2023}).
Here $\zeta(s,x)\deq\sum_{n=0}^\infty (n+x)^{-s}$ $(\Re s>1)$ is the Hurwitz zeta function,
$\kappa=1/{\sqrt{\alpha\beta-1}}$,
and each $R_{k,j}(\kappa)$ is given by a sum of integrals
\begin{equation*}
R_{k,j}(\kappa)=
\sum_{1\le i_1<i_2<\dots<i_{2j}\le k}\int_{[0,1]^k}
\frac{2^k\,du_1\dots du_k}{\sqrt{\mathcal{W}_k(\vu;\kappa;i_1,\dots,i_{2j})}},
\end{equation*}
where $\mathcal{W}_k(\vu;\kappa;i_1,\dots,i_{2j})$ is a certain
(concretely given) polynomial function in $u_1,\dots,u_k$.

Let us observe the first three special values
$\zeta_Q(2)$, $\zeta_Q(3)$ and $\zeta_Q(4)$:
\begin{align*}
\zeta_Q(2)
&=2\biggl(\frac{\alpha+\beta}{2\sqrt{\alpha\beta(\alpha\beta-1)}}\biggr)^{\!\!2}
\Biggl(\zeta(2,1/2)+\biggl(\frac{\alpha-\beta}{\alpha+\beta}\biggr)^{\!\!2}R_{2,1}(\kappa)\Biggr),
\\
\zeta_Q(3)
&=2\biggl(\frac{\alpha+\beta}{2\sqrt{\alpha\beta(\alpha\beta-1)}}\biggr)^{\!\!3}
\Biggl(\zeta(3,1/2)+\biggl(\frac{\alpha-\beta}{\alpha+\beta}\biggr)^{\!\!2}R_{3,1}(\kappa)\Biggr),
\\
\zeta_Q(4)
&=2\biggl(\frac{\alpha+\beta}{2\sqrt{\alpha\beta(\alpha\beta-1)}}\biggr)^{\!\!4}
\Biggl(\zeta(4,1/2)+\biggl(\frac{\alpha-\beta}{\alpha+\beta}\biggr)^{\!\!2}R_{4,1}(\kappa)
+\biggl(\frac{\alpha-\beta}{\alpha+\beta}\biggr)^{\!\!4}R_{4,2}(\kappa)\Biggr),
\end{align*}
where
\begin{align*}
R_{2,1}(\kappa)
&=\int_{[0,1]^2}\frac{4du_1du_2}
{\sqrt{(1-u_1^2u_2^2)^2+\kappa^2(1-u_1^4)(1-u_2^4)}},
\\
R_{3,1}(\kappa)
&=3\int_{[0,1]^3}\frac{8du_1du_2du_3}
{\sqrt{(1-u_1^2u_2^2u_3^2)^2+\kappa^2(1-u_1^4)(1-u_2^4u_3^4)}},
\\
R_{4,1}(\kappa)
&=4\int_{[0,1]^4}\frac{16du_1du_2du_3du_4}
{\sqrt{(1-u_1^2u_2^2u_3^2u_4^2)^2+\kappa^2(1-u_1^4)(1-u_2^4u_3^4u_4^4)}}
\\
&\qquad{}
+2\int_{[0,1]^4}\frac{16du_1du_2du_3du_4}
{\sqrt{(1-u_1^2u_2^2u_3^2u_4^2)^2+\kappa^2(1-u_1^4u_2^4)(1-u_3^4u_4^4)}},
\\
R_{4,2}(\kappa)
&=\int_{[0,1]^4}\frac{16du_1du_2du_3du_4}
{\sqrt{(1-u_1^2u_2^2u_3^2u_4^2)^2+\kappa^2(1-u_1^4u_2^4)(1-u_3^4u_4^4)+(\kappa^2+\kappa^4)(1-u_1^4)(1-u_2^4)(1-u_3^4)(1-u_4^4)}}.
\end{align*}
We introduce the sequence $\J kn$ called \emph{Apéry-like numbers} for NCHO by
\begin{equation}\label{General-Int-Expression}
R_{k,1}(\kappa)=\sum_{n=0}^\infty \binom{-\frac12}n\J kn\kappa^{2n},
\end{equation}
for $k\ge2$ and $n\ge0$.
It is convenient to introduce the numbers $\J0n$ and $\J1n$ by
\begin{equation*}
\J0n=0,\qquad
\J1n=\frac{2^nn!}{(2n+1)!!}
=\frac{(1)_n(1)_n}{(\frac32)_n}\frac1{n!}
\qquad(n=0,1,2,\dots),
\end{equation*}
where $(a)_n=a(a+1)\dotsb(a+n-1)$ is the Pochhammer symbol.
Then $\J kn$ satisfy the recurrence relation
\begin{equation}\label{eq:recurrence_of_Jkn}
4n^2\J kn-(8n^2-8n+3)\J k{n-1}+4(n-1)^2\J k{n-2}
=4\J{k-2}{n-1}
\end{equation}
for $k\ge2$ and $n\ge2$.
We can solve the recurrence and have a closed formula for $\J kn$ (see Subsection \ref{A1}).
Let us introduce the generating function $w_k(z)=\sum_{n=0}^\infty \J kn z^n$.
There exists a \emph{ladder structure} among these functions in the sense that
the functions $w_k(z)$ satisfy the differential equation
\begin{equation}\label{Ladder structure}
Dw_k(z)=w_{k-2}(z),\qquad
D:=z(1-z)^2\frac{d^2}{dz^2}+(1-3z)(1-z)\frac{d}{dz}+z-\frac34
\end{equation}
for $k\ge2$. Note here that $w_0(z)=0$ and $w_1(z)=\frac12{}_2F_1(1,1;\frac32;z)$ \cite{KW2006KJM}. Particularly, the integral
$R_{2,1}(\kappa)$ is explicitly calculated \cite{IW2005KJM,O2008RJ}
and we have
\begin{align}\label{eq:SV at 2}
\zeta_Q(2)
&=\biggl(\frac{\pi(\alpha+\beta)}{2\sqrt{\alpha\beta(\alpha\beta-1)}}\biggr)^{\!\!2}
\Biggl(1+\frac1{2\pi\sqrt{-1}}
\Biggl(\frac{\alpha-\beta}{\alpha+\beta}\Biggr)^{\!\!2}\int_{\abs z=r}\frac{u(z)}{z(1+\kappa^2z)^{1/2}}dz\Biggr)
\\
&=\biggl(\frac{\pi(\alpha+\beta)}{2\sqrt{\alpha\beta(\alpha\beta-1)}}\biggr)^{\!\!2}
\Biggl(1+\Biggl(\frac{\alpha-\beta}{\alpha+\beta}\Biggr)^{\!\!2}\hgf21{\frac14,\frac34}1{-\kappa^2}^{\!2}\Biggr),
\end{align}
where $u(z)=w_2(z)/3\zeta(2)$ is a normalized (unique) holomorphic solution of $Du(z)=0$
in $\abs z<1$ and $\kappa^2<r<1$.

As for $w_2(z)$ and $w_4(z)$,
we observe that these have a \emph{modular interpretation} (see \cite{KW2007, KW2023}). 
Let $\tau\in \mathfrak{H}$, $\mathfrak{H}$ being the complex upper half plane, and put $q:=e^{2\pi i \tau}$. 
\begin{enumerate}
\item
The equation $Dw_2(z)=0$ is the Picard-Fuchs equation for a certain family of elliptic curves,
and $w_2(z)$ becomes a (non holomorphic) modular form of weight $1$ with respect to the congruence subgroup $\Gamma(2)(\cong \Gamma_0(4))$
if we take $z$ as 
\begin{equation}\label{eq:def_of_t}
z=z(\tau)=-\frac{\theta_2(\tau)^4}{\theta_4(\tau)^4}
=\frac{\eta(\tau)^8\eta(4\tau)^{16}}{\eta(2\tau)^{24}} \quad (\text{Hauptmodul for}\; \Gamma(2)),
\end{equation}
where $\theta_{j}(\tau)$ $(j=2,3,4)$ are the elliptic theta functions
and $\eta(\tau)$ is the Dedekind eta function:
$$
\theta_2(\tau):= \sum_{n=-\infty}^\infty q^{(n+\frac12)^2/2},
\qquad 
\theta_3(\tau):= \sum_{n=-\infty}^\infty q^{n^2/2},
\qquad 
\theta_4(\tau):= \sum_{n=-\infty}^\infty (-1)^nq^{n^2/2},
$$
$$
\eta(\tau):=q^{1/24}\prod_{n=1}^\infty(1-q^n) = \sum_{m=-\infty}^\infty(-1)^mq^{(6m+1)^2/24}.
$$

This $z(\tau)$ is a $\Gamma(2)$-modular function such that $z(i\infty)=0$ and 
\begin{equation}\label{eq:w2}
w_2(z)=\frac{\J20}{1-z}\,\hgf21{\frac12,\frac12}{1}{\frac{z}{z-1}}\footnote{Note that the complete elliptic integral of the first kind $K(\kappa)=\int_0^1\frac{dt}{\sqrt{(1-t^2)(1-\kappa^2t^2)}}$ is equal to $\frac{\pi}2\hgf21{\frac12,\frac12}{1}{\kappa^2}$.}
=\J20{\frac{\theta_3(\tau)^4}{\theta_4(\tau)^2}}
=\J20\frac{\eta(2\tau)^{22}}{\eta(\tau)^{12}\eta(4\tau)^8}
\footnote{We correct the mistake for indices of the quotients of $\theta$-functions in p.249 in \cite{KW2023}, though the final expression by Dedekind $\eta$-functions there is correct.} .
\end{equation}

\item
If we take $z$ as above \eqref{eq:def_of_t} in the case of $w_2(z)$,
then $w_4(z)$ is expressed in terms of $w_2(z)$
and the derivative of an \emph{automorphic integral (or Eichler form)} (a slight generalization of the form studied in \cite{G1961})
for a certain congruence subgroup.
See below.
\end{enumerate}
We hope to acquire a geometric and arithmetic understanding for $w_3(z)$ 
as well as to comprehend $R_{4,2}(\kappa)$ from a similar perspective.

\subsubsection*{Differential Eisenstein series}

Let $\tau\in \mathfrak{H}$. Define
\begin{align*}
G(s,\tau)
&\deq \psum_{m,n\in\Z}(m\tau+n)^{-s},
\\
G^{(N;a,b)}(s,\tau)
&\deq \psum_{\substack{m,n\in\Z \\ m\equiv a\,(\mathrm{mod}\,N) \\ n\equiv b\,(\mathrm{mod}\,N)}}
(m\tau+n)^{-s}
\qquad(a,b\in\{0,1,\dots,N-1\})
\end{align*}
for $s\in\C$ such that $\Re(s)>2$.
Here we choose the branch of complex powers by
\begin{equation*}
a^{-s}=\exp(-s\log a),\quad -\pi\le\arg a<\pi,
\end{equation*}
and $\psum_{m,n\in\Z}$ means the sum over all pairs $(m,n)$ of integers such that the summand is defined.
We sometimes refer to these series as \emph{generalized Eisenstein series} (e.g.\ \cite{B1975Cr}).
It is known that $G(s,\tau)$ is analytically continued to the whole $s$-plane.

For $m\in\Z$, define
\begin{align*}
\dG{m}(\tau)&\coloneqq \frac{\partial}{\partial s}G(s,\tau)\bigg|_{s=m},
\\
\dG[(N;a,b)]{m}(\tau)&\coloneqq \frac{\partial}{\partial s}G^{(N;a,b)}(s,\tau)\bigg|_{s=m}\qquad(a,b\in\{0,1,\dots,N-1\}),
\end{align*}
which we call the \emph{differential Eisenstein series}.
We see that $\dG{-2k}(\tau)$ is an Eichler form of weight $-2k$ with respect to $SL_2(\Z)$
(in the sense that $\dG{-2k}(\tau+1)=\dG{-2k}(\tau)$
and $\tau^{2k}\dG{-2k}(-1/\tau)-\dG{-2k}(\tau)\in\C(\tau)$)
and $\dG[(2;0,0)]{-2k}(\tau), \dG[(2;1,1)]{-2k}(\tau)$ are
Eichler form of weight $-2k$ with respect to $\hecke$,
where $\hecke\coloneqq \langle T^2,S \rangle \, (\supset \Gamma(2))$,
$S=\mat{0 & -1 \\ 1 & 0},\, T=\mat{1 & 1 \\ 0 & 1}$
being the standard generator of $SL_2(\Z)$
(in the sense that $\dG[(2;i,i)]{-2k}(\tau+2)=\dG[(2;i,i)]{-2k}(\tau)$
and $\tau^{2k}\dG[(2;i,i)]{-2k}(-1/\tau)-\dG[(2;i,i)]{-2k}(\tau)\in\C(\tau)$).

Using the discussion in \cite{Y2008} for the inhomogeneous equation \eqref{Ladder structure}, we have
\begin{equation}\label{eq:w4}
w_4(z)=
\frac{\pi^4}2
\frac{\theta_3(\tau)^4}{\theta_4(\tau)^2}
\biggl[1+\frac1{\pi i}\frac{d}{d\tau}\Bigl\{
7\dG{-2}(\tau)+2\dG[(2;1,1)]{-2}(\tau)
\Bigr\}\biggr].
\end{equation}
This shows that $w_4(z)$ is expressed in terms of $w_2(z)$ (see \eqref{eq:w2})
and the derivative of an Eichler form $7\dG{-2}(\tau)+2\dG[(2;1,1)]{-2}(\tau)$.

\begin{rem}
The use of the notion of the Apéry-like numbers in this paper (and \cite{KW2006KJM, KW2023})
is rather different from the one in \cite{Za2009};
Actually, the motivation to introduce the Apéry-like numbers in \cite{Za2009} is
different from ours.
Nevertheless, the numbers $\J2n$ can be identified at the $\#19$
in the list of non-trivial candidates of the Apéry numbers in \cite{Za2009}. \qed
\end{rem}

\begin{rem}
In \cite{Zu2013}, the Eisenstein series of negative weight were used
for deriving a certain integral representation  i.e., a \emph{period} in the sense of \cite{KZ2001} of the $L$-value of elliptic curve.
Our formula for the expression of $w_4(z)$ can be considered to be of this type.
As well as the discussion about the generating function of $w_k(z)$
relating modular Mahler measures in \cite{KW2023},
it would be interesting to clarifying the basic situation behind these further,
e.g., from the view point of nearly holomorphic modular forms \cite{Shi1982}
and Eichler forms (see e.g.\cite{G1961}). \qed
\end{rem}

\begin{rem}\label{RemarkGeneralValue}
It is immediately clear from the expression \eqref{eq:specialvalues of zeta_Q(s)} of special values $\zeta_Q(n)$ that $R_{k,j}(\kappa)$ depends only on the variable $\kappa(:=1/\sqrt{\alpha\beta-1})$, that is, independent from the  NCHO. Therefore, in addition to the expression in \eqref{eq:specialvalues of zeta_Q(s)},  the 
 expressions \eqref{eq:w2} (or $\zeta_Q(2)$ at \eqref{eq:SV at 2}) and \eqref{eq:w4}, we may expect that $\zeta_Q(s)$ is written  as an infinite series of certain different $L$-functions. 
This expectation partially supports the reason that $\zeta_Q(s)$ would have neither Euler product nor functional equation almost certainly. 
\end{rem}
\begin{rem}
Let us recall the following double integral which is considered as a generalization of the classical elliptic ``arithmetic-geometric mean (AGM)'' integral $I_1(f)$ due to Gauss:
$$
I_2(f):=(1+f) \int\int_{\pi>\alpha > \beta>0}\frac{d\alpha d\beta}
{\sqrt{((1+f)^2-4f\cos^2\alpha)((1+f)^2-4f\sin^2\beta)}}. 
$$
Then, the Ausserlechner conjecture (2016) about the output voltage of a Hall plate asserting that 
$$
I_2(f)=I_2\Big(\frac{1-f}{1+f}\Big) \quad \text{for}\quad f \in [0,\,1]
$$
is proved in \cite{BZ2019JAMS}. A remarkable point is that the integral $I_2(f)$ has a modular interpretation which follows from the modular parameterization of the AGM and the third-order inhomogeneous diffrential equation satisfied by $I_2(f)$ by taking
\begin{align}\label{Zagier11}
f^2=16\frac{\eta(\tau)^8\eta(4\tau)^{16}}{\eta(2\tau)^{24}}.
\end{align}
This modularity study in \cite{BZ2019JAMS} is not directly related to the present discussion, but has a similar structure such as i) the choice \eqref{Zagier11} of a modular function $f^2$ (essentiallly, identical to the case \eqref{eq:def_of_t} for $w_2(z)$) is the same as the \#11 of the Ape\'ry-like recurrence equation in \cite{Za2009} and ii) the inhomogeneous term of the inhomogeneous differential equation is given by the meromorphic modular form of weight $4$ like the ladder relation \eqref{Ladder structure} for $k=4$. 
In addition, since $I_2(1)=\frac{\pi}4{}_2F_1(\frac12,\frac12; \frac32; 1)\big(=\frac{\pi^2}8\big)$ at the end point $f=1$, we observe that $I_2(1)=\frac{\pi}2w_1(1)$ with $w_1(z)=\frac12{}_2F_1(\frac12,\frac12; \frac32;z)$
(see \eqref{Ladder structure}).
\end{rem}

\subsection{Picard-Fuchs' equation for $w_2(t)$}
First we recall the integral expression of $w_2(t)$ (see \cite{KW2007}) as 
\begin{align}
    w_2(t)=4\int_0^1\int_0^1\frac{1-X^2Y^2}{(1-X^2Y^2)^2-t(1-X^4)(1-Y^4)}dXdY.
\end{align}
Put $T:=\sqrt{\frac{t}{1-t}}$ and $Q_T(X,Y):=(1-X^2Y^2)+T(X^2-Y^2)$.  We now define
$\mathcal{W}_2(T)=\frac{1-t}4 w_2(t)$. Then we observe \cite{KW2007} that
\begin{align}
\mathcal{W}_2(T)= \int_{\square}\frac{dXdY}{Q_T(X,Y)},
\end{align}
where $\square$ denotes the domain $[0,1]\times[0,1]$.
The denominator $Q_T(X,Y)$ of the integaral $\mathcal{W}_2(T)$ defines an algebraic curve $Q_T: Q_T(X,Y)=0$ in $\C$. It follows easily form the equation \eqref{Ladder structure} with $w_0(t)\equiv 0$ that $\mathcal{W}_2(T)$ satisfies the differential equation $\mathcal{L}\mathcal{W}_2(T)=0$. Here the operator $\mathcal{L}$ is given by 
\begin{align}\label{eq:PF-operator}
    \mathcal{L}:= T(T^2-1)\frac{d^2}{dT^2}+(3T^2-1)\frac{d}{dT}+T. 
\end{align}
It is known \cite{KW2007} that 
\begin{enumerate}
\item the algebraic curve $Q_T$ is birationally equivalent to a certain elliptic curve $C_T$ for all but finite values of $T$, and $\{C_T\}_{T\in \sqrt{-1}\Q^\times}$ gives the family of elliptic curves having rational 4-torsion.\footnote{The defining equation of the curve $C_T$ is given by 
$$
y^2=x^3-\frac{T^4+14T^2+1}{48}x+\frac{T^6-33T^4-33T^2+1}{864}.$$}
\item the differential equation $\mathcal{L}\mathcal{W}_2(T)=0$ is the Picard-Fuchs equation corresponding to the family $\{C_T\}_{T}$. 
\end{enumerate}
On the other hand, it is known that the $L$-value $L(E,2)$ (resp. $L(E,3)$) for a conductor $32$ elliptic curve $E$ is identified to the $L$-value $L(f,2)$ (resp. $L(f,3)$) for the cusp form $f(\tau)=\eta(4\tau)^2\eta(8\tau)^2$ (see \cite{Zu2013}). In the manipulation for evaluating $L(E,3)$
in \cite{Zu2013} (for aiming to show the periodness of the value), it is interesting to observe the appearance of $\mathcal{W}_2(T)$ (the complete elliptic integral of the first kind) 
$$
\mathcal{W}_2(T)= {}_2F_1\Big(\frac12,\frac12; 1; T^2\Big)
=\sum_{n=0}^\infty \binom{2n}{n}^{\!2} \Big(\frac{T}4\Big)^{\!2n}=
\frac{\eta(4\tau)^{10}}{\eta(2\tau)^4\eta(8\tau)^4} 
$$
for $T=4\frac{\eta(2\tau)^{4}\eta(8\tau)^8}{\eta(4\tau)^{12}}$. 
Although it is necessary to make more investigation, the point is that the special value of $\zeta_Q(2)$ has some connection with an elliptic curve $E$ of conductor $32$.

\begin{rem}
As we observe that the differential equation satisfies by $w_2(z)$ is esssentially the Picard-Fuchs eqation for the family of elliptic curves. In addition, since it has the expression \eqref{eq:w2}, we may define the Dirichlet series $L(w_2, s)$ associalted with the modular form $w_2$. It raises a question of whether there exists an elliptic curves $C$ such that $L(C, s)$ esssentially coincides with $L(w_2, s)$. Partially, it might be the fist step to approach the expectation described in Remak \ref{RemarkGeneralValue}.
\end{rem}

\subsection{Summary on special values and associated Apéry(-like) numbers}

For the special values $\zeta(2)$ and $\zeta(3)$,
Apéry defined rational numbers $A_k(n)$ and $B_k(n)$ $(k=2,3)$ to prove their irrationality.
They satisfy similar but different three-term recurrence relations respectively,
and hence their generating functions satisfy
corresponding second order ordinary differential equations.
The differential equation for $\Rpery k(t)$ has an interpretation
as a Picard-Fuchs equation for certain curves/surfaces and, thus,
both $\Rpery2(t)$ and $\Rpery3(t)$ have modular interpretation. 
Notice that we have \emph{not} found Apéry numbers for other special values $\zeta(4), \zeta(5), \ldots$
playing a similar role so far.

For \emph{every} special value $\zeta_Q(k)$ $(k=2,3,4,\ldots)$,
we defined the numbers $\J kn$ from the term $R_{k,1}(\kappa)$
in the formula for $\zeta_Q(k)$.
They satisfy three-term recurrence relations;
In this case, the relations between three terms are \emph{common},
but the inhomogeneous term (right hand side) is \emph{another Apéry-like number of lower order}.
As a result, the generating functions for Apéry-like numbers satisfy
an ordinary differential equations for a common second order differential operator
with another generating function for Apéry-like numbers as an inhomogeneous term.
Though the generating function $w_2(t)$ for $\J2n$ has a modular interpretation,
we have not obtained such an interpretation for $w_k(t)$, $k\ge3$.
Nevertheless, for $w_4(t)$, we have an expression in terms of
differential Eisenstein series, which is an Eichler form with respect to $\hecke$.
These are summarized in Table \ref{PFEq} below.

\begin{table}[htb]
\centering
\begin{tabular}{ccc}
special value & Picard-Fuchs equation & modular interpretation
\\
\hline
$\zeta(2)$ & elliptic curves with rational 5-torsion & $\Gamma_1(5)$-modular form of weight $1$ \\
$\zeta(3)$ & K3 surfaces with Picard number 19 & $\Gamma_1(6)$-modular form of weight $2$ \\
\hline
$\zeta_Q(2)$ & elliptic curves with rational 4-torsion & $\Gamma(2)$-modular form of weight $1$ \\
$\zeta_Q(3)$ & inhomogeneous analog for $\zeta_Q(2)$ case & unknown \\
$\zeta_Q(4)$ & inhomogeneous analog for $\zeta_Q(2)$ case & derivative of a $\hecke$-Eichler form of weight $-2$ \\
\end{tabular}
\caption{Picard-Fuchs equations for algebraic varieties,
modular interpretation for the generating function of associated Apéry(-like) numbers
for special values of $\zeta(s)$ and $\zeta_Q(s)$.}\label{PFEq}
\end{table}


\section{Quantum Rabi models (QRM)}
\label{sec:QRM}

In this section, we briefly introduce the partition function of the \emph{quantum Rabi model} (QRM)
and some general aspects of its spectral zeta function and related (Rabi-)Bernoulli numbers.
The main reason for considering it is that, as mentioned in the introduction,
the NCHO may be considered as a covering model of the QRM \cite{W2016, RW2023CMP}
by a certain confluence process of regular singular points of the Heun ODE.
Moreover, it is an important point of comparison since, different to the NCHO,
the partition function of the QRM is explicitly known \cite{RW2021}.

The QRM is widely recognized as the simplest and most fundamental model
describing quantum light-matter interactions,
that is, the interaction between a two-level system and a bosonic field mode.
Indeed, it is considered as a milestone in the long history of quantum physics \cite{HR2008, JC1963}.
We direct the reader to \cite{BCBS2016JPA} for a recent collection of introductory,
survey and original articles from both experimental and theoretical viewpoints,
not limited to light-matter interaction but also in other fields of research.
A notable achievement in recent experimental studies is presented in \cite{YS2018}.
The Hamiltonian $\HRabi$ of the QRM is precisely given by
\[
\HRabi \deq \omega a^{\dagger}a + \Delta \sigma_z + g (a + a^{\dagger}) \sigma_x .
\]
Here, \(a^{\dagger}\) and \(a\) are the creation and annihilation operators
of the single bosonic mode (\([a,a^{\dagger}]=1 \)),
$\sigma_x, \sigma_z$ are the Pauli matrices (also written as \(\sigma_1\) and \(\sigma_3\),
but since there is no risk of confusion with the variable \(x\) we use the usual notations),
$2\Delta$ is the energy difference between the two levels
and $g$ denotes the coupling strength between the two-level system
and the bosonic mode with frequency $\omega$
(subsequently, we set $\omega=1$ without loss of generality). 

Let $\PRabi(t)$ $(t>0)$ be the partition function of the QRM.
Then, the Hurwitz-type spectral zeta function is given by
\begin{align}
\label{eq:speczetamellin}
\zeta_{\mathrm{QRM}}(s;\tau)
\deq \sum_{j=1}^\infty (\mu_j +\tau)^{-s}
= \frac1{\Gamma(s)}\int_0^\infty t^{s-1} \PRabi(t)e^{-t\tau}dt,
\end{align}
where $\mu_i$ are the (ordered) eigenvalues in the spectrum of $\HRabi$.
We assume that $\tau>g^2+\Delta$ so that $\mu_1+\tau>0$ \cite{S2016NMJ}. It is known that the multiplicity of each eigenvalue is less than or equal to $2$ (see, e.g. \cite{KRW2017}). Let $K_{\mathrm{QRM}}(x,y,t)$ be the heat kernel of the QRM. We refer the reader to \cite{RW2019} for the detailed derivation of the closed formula of $K_{\mathrm{QRM}}(x,y,t)$ based on the (Lie-)Trotter-Kato product formula, calculation of Gaussian integrals and use of the Fourier analysis on $\mathbb{F}_2^n\, (n=0,1,2,\ldots)$, $\mathbb{F}_2\,(\cong \Z_2)$ being the field of two elements, (the quantum Fourier transform) from the viewpoint of the Segal-Shale-Weil representation of $SL_2(\mathbb{F}_2)$\footnote{The quantum Fourier transform and Segal-Shale-Weil (or metaplectic) representation for $SL_2(\mathbb{F}_2)$ are used to the construction of quantum error correction, especially the stabilizer codes. On the other hand, the QRM is a theoretical model in quantum optics and gives rather a basis for creating devices’ side. It is interesting that the same mathematical technique is useful in both sides but in a different manner.}. 
Then,  
\begin{equation*}
\PRabi(t) = \Tr K_{\mathrm{QRM}}(t) = \int_{-\infty}^\infty \tr K_{\mathrm{QRM}}(x,x,t) dx
\end{equation*} 
gives the following expression. 

\begin{lem}[Corollary 4.3 of \cite{RW2021}] \label{cor:Partition_function}
The partition function \( \PRabi(t)\) of the QRM is given by
\begin{align*}
\PRabi(t)
&= \frac{e^{tg^2}}{1-e^{-t}} \Biggl[
1 + e^{-2g^2 \coth\frac{t}2} \sum_{\lambda=1}^{\infty} (t \Delta)^{2\lambda}
\\
&\qquad\qquad
\idotsint\limits_{0\leq \mu_1 \leq \cdots \leq \mu_{2 \lambda} \leq 1}
\exp\Bigl(
4g^2\frac{\cosh(t(1-\mu_{2\lambda}))}{\sinh(\beta)}
+ \xi_{2 \lambda}(\bm{\mu_{2\lambda}},t)
+ \psi^-_{2 \lambda}(\bm{\mu_{2 \lambda}},t)
\Bigr)
d \bm{\mu_{2 \lambda}}
\Biggr],
\end{align*}
where the functions $\xi_\lambda(\bm{\mu_{\lambda}},t)$
and $\psi_\lambda^{-}(\bm{\mu_{\lambda}},t)$ are given by
\begin{align*}
\xi_\lambda(\bm{\mu_{\lambda}},t)
&\deq -\frac{8g^2}{\sinh t}
\bigl(\sinh \tfrac12t(1-\mu_\lambda) \bigr)^{2}
(-1)^{\lambda}\sum_{\gamma=0}^{\lambda} (-1)^{\gamma} \cosh( t \mu_{\gamma}) \nonumber
\\
& \qquad{} - \frac{4 g^2}{\sinh t}
\sum_{\substack{0\leq\alpha<\beta\leq \lambda-1\\ \beta - \alpha \equiv 1 \pmod{2}}}\bigl( \cosh(t(\mu_{\beta+1}-1))-\cosh(t(\mu_{\beta}-1)) \bigr) 
( \cosh(t \mu_{\alpha}) - \cosh(t \mu_{\alpha+1})), \nonumber
\\
\psi_\lambda^{-}(\bm{\mu_{\lambda}},t)
&\deq \frac{4 g^2}{\sinh t}\biggl[
\sum_{\gamma=0}^{\lambda} (-1)^{\gamma}
\sinh(t\bigl(\tfrac12 - \mu_{\gamma})\bigr)
\biggr]^2,
\end{align*}
for \(\lambda \geq 1\) and \(\bm{\mu_{\lambda}} = (\mu_1,\mu_2,\ldots,\mu_\lambda) \) and \( \mu_0 = 0 \).
Note that each term of the series is positive (see \cite{RW2023Limit}). 
\end{lem}

We define the function $\Omega(t) = \Omega(t;\Delta,g)$ implicitly
by the equation $\PRabi(t) = \frac{\Omega(t)}{1-e^{-t}}$. 
Concretely, \(\Omega(t) \) is given by
\begin{multline*}
\Omega(t) \deq 2 e^{g^2 t} \Biggl[
1 + \sum_{\lambda=1}^{\infty} (t \Delta)^{2\lambda}
\\
\idotsint\limits_{0\leq \mu_1 \leq \cdots \leq \mu_{2 \lambda} \leq 1}
\exp\Bigl(-2g^2 \coth\frac{t}2+ 4g^2\frac{\cosh(t(1-\mu_{2\lambda}))}{\sinh(t)} + \xi_{2 \lambda}(\bm{\mu_{2\lambda}},t) + \psi^-_{2 \lambda}(\bm{\mu_{2 \lambda}},t)\Bigr)
d \bm{\mu_{2 \lambda}}
\Biggr].
\end{multline*}

Then the function $\Omega(t)$ extends holomorphically for complex $t$ as follows.

\begin{lem}[Proposition A.1 of \cite{RW2021}] 
The series defining the function $\Omega(t)$ is uniformly convergent in compacts
in the complex domain $\mathcal{D}$ consisting of the union of the half plane $\Re t>0$
and a disc centered at origin with radius $r < \pi$.
In particular, $\Omega(t)$ is a holomorphic function in the region \(\mathcal{D}\).
\end{lem}

Using the Mellin transform expression \eqref{eq:speczetamellin} of $\zeta_{\mathrm{QRM}}(s)$
by the partition function $Z_{\mathrm{QRM}}(t)$
we have another proof for the meromorphic continuation,
similar to one of Riemann's original proofs for the zeta function. 

\begin{thm}[Theorem 4.1 of \cite{RW2021}] \label{IntRep_SZF} 
We have
\begin{align}
\zeta_{\mathrm{QRM}}(s;\tau)
=-\frac{\Gamma(1-s)}{2\pi i}\int_\infty^{(0+)} \frac{(-w)^{s-1}\Omega(w)e^{-\tau w}}{1-e^{-w}}dw.
\end{align}
Here the contour integral is given by the path which starts at $\infty$ on the real axis,
encircles the origin (with a radius smaller than $2\pi$) in the positive direction
and returns to the starting point
and it is assumed $\abs{\arg(-w)}\leq \pi$.
This gives a meromorphic continuation of $\zeta_{\mathrm{QRM}}(s;\tau)$ to the whole plane
where the only singularity is a simple pole with residue $2$ at $s=1$.
\end{thm}

We now define the \emph{$k$-th Rabi-Bernoulli polynomials} $(RB)_k(\tau, g^2, \Delta)$
(according to the naming in \cite{S2016NMJ} and \cite{RW2021}, see Remark \ref{RBP} below).
Notice that when $\Delta=0$,
the $k$-th Rabi-Bernoulli polynomial is equal to the $g^2$-shift $B_k(\tau-g^2)$
of the $k$-th Bernoulli polynomial $B_k(\tau)$.

\begin{dfn}\label{Rabi-Bernoulli}
The $k$-th Rabi-Bernoulli polynomial $(RB)_k(\tau, g^2, \Delta^2) \in \R[\tau, g^2, \Delta^2]$ is defined
through the equation
\begin{equation*}
\w \PRabi(\w)e^{-\tau \w}
=\frac{\w\Omega(\w)e^{-\tau \w}}{1-e^{-\w}}
\eqqcolon 2 \sum_{k=0}^\infty\frac{(-1)^k(RB)_k(\tau, g^2, \Delta^2)}{k!}\w^k.
\end{equation*}
\end{dfn}

The special values of the spectral zeta functions at the negative integers are described as in the usual way.

\begin{lem}\label{SVatNegative}
We have, for \(k \geq 1 \),
\begin{align*}
\zeta_{\mathrm{QRM}}(1-k;\tau) = -\frac2k (RB)_k(\tau, g^2, \Delta^2).
\end{align*}
\end{lem}

\begin{rem} \label{RBP}
From the expression $\zeta_{\mathrm{QRM}}(1-k;\tau)$ above,
we find that $(RB)_k(\tau, g^2, \Delta^2)$ is identical
with the Rabi-Bernoulli polynomials $R_k(g,\Delta,\tau)$ in (1.1) of \cite{S2016NMJ}:
\[
R_k(g,\Delta,\tau)= (RB)_k(\tau, g^2, \Delta^2).
\]
Also, $(RB)_k(\tau, 0, 0)$ is equal to the Bernoulli polynomial $B_k(\tau)$.
The coefficient $2$ appearing in the definition of the Rabi-Bernoulli polynomials
is the effect of the QRM being a two-by-two system Hamiltonian. \qed
\end{rem}

\begin{lem}\label{cor:DD-relation}
We have 
\[
\frac{\partial^n}{\partial \tau^n} \zeta_{\mathrm{QRM}}(s;\tau)
= (-1)^n (s)_n \zeta_{\mathrm{QRM}}(s+n;\tau),
\]
where \((a)_n = a (a+1) \cdots (a+n-1)=\frac{\Gamma(a+n)}{\Gamma(a)} \) is the Pochhammer symbol. \qed
\end{lem}

The following simple difference-differential equation satisfied by the Rabi-Bernoulli polynomials,
similar to the one for the Bernoulli polynomials,
is a consequence of Lemmas \ref{SVatNegative} and \ref{cor:DD-relation}.

\begin{lem} 
We have 
\begin{equation}\label{DD-equation}
\frac{\partial}{\partial \tau}(RB)_{k+1}(\tau, g^2, \Delta^2)
= -(k+1) (RB)_k(\tau, g^2, \Delta^2)
\end{equation} 
for $k=0,1,2,\ldots$.
\qed
\end{lem}

The explicit formula of \(\Omega(t)\) allows us to give another proof to the rationality
of the coefficients of the Rabi-Bernoulli polynomials \((RB)_k(\tau, g^2, \Delta^2)\), 
proved originally in \cite{S2016NMJ}.

\begin{thm}[Theorem A.5 of \cite{RW2021}] \label{thm:Rationality_RB}
The Rabi-Bernoulli polynomials \((RB)_k(\tau, g^2, \Delta^2)\),
as polynomials in \(\tau,g^2\) and \(\Delta^2\), have rational coefficients.
That is, \((RB)_k(\tau, g^2, \Delta^2) \in \Q[\tau, g^2, \Delta^2]\).
Furthermore, the $k$-th Rabi-Bernoulli polynomials \((RB)_k(\tau, g^2, \Delta^2)\)
with respect to the variable $\tau$ is monic and its degree is exactly equal to $k$.
\end{thm}

\begin{ex}
We give here few examples of Rabi-Bernoulli polynomials.
It is easy to compute these polynomials (at $\tau=0$)
directly from the series expansion of the partition function $\PRabi(\w)=\Omega(\w)/(1-e^{-\w})$.
Note first that 
\[
\w\PRabi(\w)=\frac{\w\Omega(\w)}{1-e^{-\w}}
= \Omega(\w)\biggl[
1 + \Bigl(\frac{\w}2 - \frac{\w^2}6 + \cdots\Bigr)
+ \Bigl(\frac{\w}2 - \frac{\w^2}6 + \cdots\Bigr)^{\!2} + \cdots
\biggr] 
\]
by taking small enough $\w$.
Since $\Omega(0)=2$, using integration (due to the relation \eqref{DD-equation}) and observing the first few terms of the expansion of $\Omega(\w)$ at $\w=0$ gives
\begin{align*}
(RB)_0(\tau, g^2, \Delta)
&=1,
\\
(RB)_1(\tau, g^2, \Delta)
&= \tau - \frac12 - g^2 \,
\Bigl(= \frac12\frac{\partial}{\partial \tau}(RB)_2(\tau, g^2, \Delta)\Bigr),
\\
(RB)_2(\tau, g^2, \Delta)
&= \tau^2-(1+2g^2)\tau+\frac16+g^2+g^4+\Delta^2. 
\end{align*}
\end{ex}

\begin{rem}
A generalization of the QRM, known as the asymmetric quantum Rabi model (AQRM), is defined by the Hamiltonian 
\[
\HRabi^{\e} \deq \omega a^{\dagger}a + \Delta \sigma_z + g (a + a^{\dagger}) \sigma_x + \e \sigma_x. 
\]
for any bias parameter $\e\in \R$.
Despite the apparent simplicity of the definition,
the AQRM is a model of considerable mathematical and physical
(theoretical and experimental \cite{YS2018}) significance.
The details for the discussion above on the spectral zeta function and partition function for the AQRM
can be found in \cite{R2023}.
\qed
\end{rem}

\section{Quasi-partition functions}
\label{sec:quasiPF}

In this section, we introduce the notion of ``quasi-partition function''
of a (quantum) Hamiltonian system using the special values
of the corresponding spectral zeta function at negative integers.
In order to see that the definition of the quasi-partition function is natural,
we recall the cases of the quantum harmonic oscillator and the quantum Rabi model,
where it coincides with the partition function.
Then, considering this notion for the NCHO and in view of certain bounds of the eigenvalues,
we state the conjecture that the quasi-partition function is identified to the partition function of the NCHO. 


Let $H$ be the Hamiltonian of the quantum interaction system.
We assume that $H$ is self-adjoint and only discrete eigenvalues 
\[
(0<)\, \lambda_1\leq \lambda_2\leq \lambda_3\leq \dots(\nearrow\infty)
\]
and the multiplicity of $\lambda_j$ is uniformly bounded.
Moreover, we assume that the corresponding spectral zeta function
$\zeta_H(s)=\sum_{j=1}^\infty \lambda_j^{-s}$,
and more generally, the Hurwitz-type spectral zeta function
$\zeta_H(s, \tau)=\sum_{j=1}^\infty (\lambda_j+\tau)^{-s}$,
converges absolutely for $\Re(s)\gg0$
and has a meromorphic extension to the whole complex plane $\C$
having unique simple pole at $s=1$ (for the sake of a technical requirement).
Examples include the QRM, AQRM, Jaynes-Cummings model \cite{JC1963},
and certain second order semiregular non-commutative harmonic oscillators \cite{MM2023}.

\begin{dfn}\label{Quasi_PF} 
Under the assumption on $H$ above,
we define the \emph{quasi-partition function} $\tZ_H(t)$ of the system defined by $H$ as follows
if the series converges for small $t>0$: 
\begin{equation*}
\tZ_H(t) \deq \sum_{k=0}^\infty (-1)^k \frac{\zeta_H(-k)}{k!}t^k +\Res_{s=1}\zeta(s)t^{-1}.
\end{equation*}
\end{dfn}

\begin{ex}\label{qHO}
Consider the quantum harmonic oscillator defined by $Q_{\mathrm{qHO}}$.
Since the spectrum of $Q_{\mathrm{qHO}}$ is given by the positive half integers
$\{n+\frac12\}_{n=0,1,2,\ldots}$,
the spectral zeta function $\zeta_{\mathrm{qHO}}$ is given by the Hurwitz zeta function
$\zeta(s, \tau)(s)= \sum_{n=0}^\infty(n+\tau)^{-s}$ $(\Re(s)>1)$.
Actually,
\[
\zeta_{\mathrm{qHO}}(s)
=\sum_{n=0}^\infty \Bigl(n+\frac12\Bigr)^{\!-s}
=\zeta\Bigl(s, \frac12\Bigr)
\bigl(=(2^s-1)\zeta(s)\bigr).
\]
Note that $\zeta_{\mathrm{qHO}}(-n)=\zeta(-n, \frac12)=-\frac{B_{n+1}(\frac12)}{n+1}$,
we observe that
\begin{align*}
\tZ_{Q_{\mathrm{qHO}}}(t)
&= \sum_{k=0}^\infty (-1)^k \frac{\zeta_{\mathrm{qHO}}(-k)}{k!}t^k+ t^{-1}
=\sum_{k=0}^\infty (-1)^{k+1}\frac{B_{k+1}(\frac12)}{(k+1)!}t^k+ t^{-1}
\\
&= -\sum_{k=0}^\infty (-1)^k\frac{B_{k}(\frac12)}{k!}t^{k-1}
=\frac{e^{-\frac{t}2}}{1-e^{-t}}
\\
&=\sum_{n=0}^\infty e^{-t(n+\frac12)}
=Z_{Q_{\mathrm{qHO}}}(t). \qed
\end{align*}
\end{ex}

\begin{ex}\label{QRM}
Consider the shifted quantum Rabi model defined by the Hamiltonian $\HRabi+\tau$.
By definition, the quasi-partition of this system denoted by $\tZ_{\mathrm{QRM}+\tau}(t)$ is given by 
\[
\tZ_{\mathrm{QRM}+\tau}(t)
=\sum_{k=0}^\infty(-1)^k\frac{\zeta_{\mathrm{QRM}}(-k, \tau)}{k!}t^{k} + 2t^{-1}.
\]
By Lemma \ref{SVatNegative}, we have, for \(k \geq 0\),
\begin{align*}
\zeta_{\mathrm{QRM}}(-k;\tau)
= -\frac2{k+1} (RB)_{k+1}(\tau, g^2, \Delta^2).
\end{align*}
It follows that
\begin{align*}
\tZ_{\mathrm{QRM}+\tau}(t)
&= 2\sum_{k=0}^\infty(-1)^{k+1}\frac{(RB)_{k+1}(\tau, g^2, \Delta^2)}{(k+1)!}t^{k}+ 2t^{-1}
\\
&= 2\sum_{k=0}^\infty(-1)^{k}\frac{(RB)_{k}(\tau, g^2, \Delta^2)}{(k)!}t^{k-1}
= Z_{\mathrm{QRM}+\tau}(t). 
\end{align*}
It is not difficult to see that this statement holds also for the asymmetric quantum Rabi model
(see \cite{R2023, RW2023Limit}) in a similar way. \qed
\end{ex}

These two examples show that the quasi-partition function is actually equal
to the partition function of the respective quantum system.
Thus, let us now consider the case of the NCHO.
It follows from Theorem \ref{MT_NCHO} that
the quasi-partition function $\tZ_Q(t)$ of the NCHO can be formally defined by 
\begin{equation}\label{qpf_Q}
\tZ_Q(t) \deq \sum_{k=0}^\infty (-1)^k \frac{\zeta_Q(-k)}{k!}t^k
+\frac{\alpha+\beta}{\sqrt{\alpha\beta(\alpha\beta-1)}}t^{-1}.
\end{equation}

For the coming discussion,
we now recall the following estimate for the lower and upper bounds of the eigenvalues $\lambda_j$. 
\begin{lem}[Theorem 1.2 of \cite{IW2007}] \label{leyLemma}
Let $\lambda_{2j-1}, \lambda_{2j}$ $(j=1,2,\ldots)$ be the $(2j-1)$-th and $2j$-th eigenvalues of $Q$.
Then we have 
\begin{align}
\Bigl(j-\frac12\Bigr)\min\{\alpha, \beta\}\sqrt{1-\frac1{\alpha\beta}}
\leq \lambda_{2j-1}
\leq \lambda_{2j}
\leq
\Bigl(j-\frac12\Bigr)\max\{\alpha, \beta\}\sqrt{1-\frac1{\alpha\beta}}.
\end{align}
\end{lem}

This indicates the partition function
$Z_Q(t)\deq \sum_{j=1}^\infty e^{-\lambda_j t}$ of $Q$ is bounded
by essentially the quantum harmonic oscillator from below and above.
In fact, it follows from Lemma \ref{leyLemma} above, for $t>0$
\begin{align}\label{PF-qHO}
\frac{2e^{-\frac{t}2 \max\{\alpha, \beta\} \sqrt{1-\frac1{\alpha\beta}}}}
{1- e^{-t\max\{\alpha, \beta\} \sqrt{1-\frac1{\alpha\beta}}}}
\leq
Z_Q(t)
\leq
\frac{2e^{-\frac{t}2 \min\{\alpha, \beta\} \sqrt{1-\frac1{\alpha\beta}}}}
{1- e^{-t\min\{\alpha, \beta\} \sqrt{1-\frac1{\alpha\beta}}}}.
\end{align}
Notice that by the Stirling approximation $n!\sim \sqrt{2\pi n}\bigl(\frac{n}e\bigr)^{n}$ with \eqref{SB-Even},
it leads that
\begin{equation}\label{BernulliBehavior1}
\abs{B_{2m}}\sim 4\sqrt{\pi m}\Bigl(\frac{m}{\pi e}\Bigr)^{\!2m}. 
\end{equation}
This shows
\begin{equation}\label{BernulliBehavior2}
\frac{\abs{B_{2m}}}{(2m)!} \sim 2 (2\pi)^{-2m}.
\end{equation}
Namely, the Weierstrass M-test says that
the Taylor expansion of the right side (also left) of \eqref{PF-qHO} $\times t$
is absolutely convergent for small $\abs{t}$.
Hence, in particular, we strongly expect that
the function $tZ_Q(t)$ defined in $t\geq0$ has the Taylor series expansion around the origin.
Unfortunately, despite of the reasonable bounds \eqref{PF-qHO}, we can not conclude it immediately. 

\begin{conject}\label{QPF_NCHO}
The following two are reasonably expected. 
\begin{enumerate} 
\item The quasi-partition function $\tZ_Q(t)$ gives the partition function $Z_Q(t)$ for the NCHO.
\item Under the statement (1), the function $tZ_Q(t)$ possesses the Taylor expansion at the origin.
Namely, let us write $Z_Q(t)$ as
\[
Z_Q(t)= 2\sum_{k=0}^\infty (-1)^k \frac{B_{Q,k}}{k!}t^{k-1}.
\]
The numbers $B_{Q,k}$ have the same asymptotic behavior as the Bernoulli polynomials 
$B_k(\tau)$ at $\tau=\frac12$ when $k\to \infty$. 
\end{enumerate}
\qed
\end{conject} 

\begin{rem}
Although the amount of variation in $B_k(\tau)$ between $0$ and $1$ gets quite large for higher $k$, since
$\zeta(-n,\tau)=-\frac{B_{n+1}(\tau)}{n+1}$, it follows that
\[
\frac{B_{n+1}(\frac12)}{n+1} = -\zeta\Bigl(-n,\frac12\Bigr)
= -(2^{-n}-1)\zeta(-n) = (2^{-n}-1)\frac{B_{n+1}}{n+1}.
\]
This shows the behavior of $B_k(\frac12)$ is identical to $-B_k$ when $k\to \infty$. \qed
\end{rem}

We call the coefficients $B_{Q,k}$ the $k$-th \emph{NC-Bernoulli numbers}.
We \emph{assume} that Conjecture \ref{QPF_NCHO} is true for the rest of this section.
Then, the function $tZ_Q(t)$ possesses the Taylor expansion at the origin.
By Conjecture \ref{QPF_NCHO} (1),
$tZ_Q(t)$ is holomorphic around the origin and neighborhood of the positive real axis.
Therefore, since $\zeta_Q(s)$ is expressed by the Mellin transform of $Z_Q(t)$ as
\[
\zeta_Q(s)= \frac1{\Gamma(s)}\int_0^\infty t^{s-1}Z_Q(t)dt, 
\]
the standard argument gives the following.

\begin{lem}
Assume Conjecture \ref{QPF_NCHO} holds. We have 
\begin{align}\label{ContourIntegral}
\zeta_Q(s)= -\frac{\Gamma(1-s)}{2\pi i}\int_\infty^{(0+)} (-w)^{s-1} Z_Q(w)dw.
\end{align}
Here the contour integral is given by the path which starts at $\infty$ on the real axis,
encircles the origin (with a sufficiently small radius) in the positive direction
and returns to the starting point and it is assumed $\abs{\arg(-w)}\leq \pi$.
This gives a meromorphic continuation of $\zeta_Q(s)$ to the whole plane
where the only singularity is a simple pole with residue $2B_{Q,0}$ at $s=1$.
Moreover, we have 
\begin{equation}\label{Q-negative-values2}
\zeta_Q(1-k)= -\frac2{k}B_{Q,k} \quad (k=1,2,\ldots).
\end{equation}
\qed
\end{lem}

We then observe from \eqref{Q-negative-values1} and \eqref{Q-negative-values2} that 
\begin{equation}\label{BC}
C_{Q,m}= \frac{\zeta_Q(1-2m)}{(2m-1)!} = - \frac2{(2m)!}B_{Q,2m}.
\end{equation}

By Conjecture \ref{QPF_NCHO} (2), 
since the $n$-th NC-Bernoulli number has a similar behavior
with the $n$-th Bernoulli number \eqref{BernulliBehavior2} when $n\to \infty$, we have the following.

\begin{cor}
If we assume that Conjecture \ref{QPF_NCHO} (2),
we can take $N\to \infty$ at the formula of $\zeta_Q(s)$
which gives its meromorphic continuation in Theorem \ref{MT_NCHO}.
Precisely, there exists an entire function $Z_{Q, \infty}(s)$ such that
\begin{equation}\label{zeta_Q(s)}
\zeta_Q(s)=\frac1{\Gamma(s)}\Biggl[
\frac{\alpha+\beta}{\sqrt{\alpha\beta(\alpha\beta-1)}} \frac1{s-1}
+ \sum_{j=1}^\infty \frac{C_{Q,j}}{s+2j-1} +Z_{Q, \infty}(s)
\Biggr],
\end{equation}
where the series converges absolutely and uniformly in any compact set
avoiding the negative integer points.
\qed 
\end{cor}

Moreover, in general
\begin{conject}
Under the assumption of Definition \ref{Quasi_PF},
the quasi-partition function $\tZ_H(t)$ of the quantum system defined by the Hamiltonian $H$
gives the partition function $Z_H(t)$ of the system. 
\qed
\end{conject}

\begin{rem}
We know that $\zeta_Q(s)$ is meromorphically continued to the whole plane $\C$,
and $C_{Q,j}$ are related to $B_{Q,j}$ via \eqref{BC}.
Hence, under the assumption of Conjecture \ref{QPF_NCHO},
we see that the series in the right hand of \eqref{zeta_Q(s)} converges absolutely and uniformly on compacts
so that the holomorphy of $Z_{Q,\infty}(s)$ follows. \qed
\end{rem}

\begin{rem}
When $\alpha=\beta=\sqrt2$, we know that $\zeta_Q(s)=2(2^s-1)\zeta(s)$ 
(that is, $Q$ is unitarily equivalent to a couple of the qHO), whence it follows that  
\begin{equation*}
C_{Q,j}=\frac{2(2^{1-2j}-1)\zeta(1-2j)}{(2j-1)!}
=\frac{-2(2^{1-2j}-1)B_{2j}}{(2j)!}.
\end{equation*}
Then \eqref{zeta_Q(s)} becomes
\begin{equation*}
\zeta(s)=\frac1{(2^s-1)\Gamma(s)}\Biggl[
\frac1{s-1}
-\sum_{j=1}^\infty \frac{(2^{1-2j}-1)B_{2j}}{(2j)!}\frac{1}{s+2j-1} + \frac12 Z_{Q, \infty}(s)
\Biggr].
\end{equation*}
\end{rem}

\begin{rem}
Since the NCHO (resp. $\eta$-NCHO) can be regarded as a covering model of the QRM
(resp. AQRM equipping the bias $\eta \in \R$, \cite{RW2023CMP}),
the comparison of their respective spectral zeta functions
(and partition functions) seems to be an interesting problem.
\qed
\end{rem}

\section{Divergent series expressing the Hurwitz zeta function}
\label{sec:DivergentSeries}

In this section, we observe the formal power series expression of the special values
at positive integer $n\geq 2$ for the Hurwitz zeta function $\zeta(s,\tau)$. 
First, we shortly describe the motivation for the study and computation.
For the NCHO, QRM, etc. in general,
we cannot expect any functional equation for the corresponding spectral zeta function
like the Riemann zeta function $\zeta(s)$.
As is well known, thanks to the functional equation,
the special value at the positive even integer $2n$ of $\zeta(s)$ is expressed
by the $n$-th Bernoulli number (in other words, by the special value $\zeta(1-2n)${\color{red};}
see Section \ref{sec:SV_RZ}).
Moreover, even in the case of $\zeta(s)$,
the special values at odd integer points are still mysterious
(see \cite{Zu2001} for the best knowledge about this).
Therefore, it is desirable to find any relation between positive and negative integer points if any.
In this section, we observe the formal power series expression of the special values
at positive integer $n\geq 2$ for the Hurwitz zeta function $\zeta(s,\tau)$. 
 
For $\tau>0$, $n\ge2$ and $j\ge1$, define
\begin{align}
A^{(n)}_j(\tau)
&\deq \int_0^\infty dt_{n-1} \int_{t_{n-1}}^\infty dt_{n-2}
\,\dotsb\, \int_{t_1}^\infty dw\,w^{j-1}e^{-\tau w}.
\end{align}

\begin{lem} 
For $\tau>0$, $n\ge2$ and $j\ge1$, we have
\begin{equation}\label{An}
A^{(n)}_j(\tau)=\frac{(j+n-2)!}{(n-1)!}\times\frac1{\tau^{j+n-1}}.
\end{equation}
\end{lem}

\begin{proof}
Put
\begin{equation}
D \deq \set{(w,t_1,\dots,t_{n-2},t_{n-1})\in \R^n}{0\le t_{n-1}\le t_{n-2}\le\dots\le t_1\le w}.
\end{equation}
Then we have
\begin{equation}
A^{(n)}_j(\tau)
=\int\dotsb\int_D w^{j-1}e^{-\tau w}\,dwdt_1\dotsb dt_{n-1}.
\end{equation}
Let us define the regions
\begin{align*}
\Omega_0
&\deq \set{(s_1,\dots,s_{n-1})}{s_1,\dots,s_{n-1}\ge0,\; 0\le s_1+\dotsb+s_{n-1}\le 1},
\\
\Omega
&\deq \set{(rs_1,\dots,rs_{n-1},r(1-s_1-\dotsb-s_{n-1}))}{r\ge0,\;(s_1,\dots,s_{n-1})\in\Omega_0}
\end{align*}
and consider the change of variables given by
\begin{equation}\label{change of variables}
w=r,\quad
t_1=r(1-s_1),\quad
t_2=r(1-s_1-s_2),\quad
\dots,\quad
t_{n-1}=r(1-s_1-\dotsb-s_{n-1}).
\end{equation}
The Jacobian of \eqref{change of variables} is
\begin{equation*}
\begin{vmatrix}
1 & 0 & 0 & \dots & 0 \\
1-s_1 & -r & 0 & \dots & 0 \\
1-s_1-s_2 & -r & -r & \dots & 0 \\
\vdots & \vdots & \vdots & \ddots & \vdots \\
1-s_1-\dotsb-s_{n-1} & -r & -r & \dots & -r
\end{vmatrix}
=(-r)^{n-1}.
\end{equation*}
Hence we have
\begin{align*}
A^{(n)}_j(\tau)
&=\int\dots\int_{\Omega} r^{j-1}e^{-\tau r}\,r^{n-1}drds_1\dotsb ds_{n-1}
\\
&=\int_0^\infty r^{n+j-2}e^{-\tau r}\,dr\times\vol(\Omega_0)
\\
&=\frac1{\tau^{n+j-1}}\Gamma(n+j-1)\times\frac1{(n-1)!}
=\frac{(n+j-2)!}{(n-1)!}\times\frac1{\tau^{n+j-1}}.
\qedhere
\end{align*}
\end{proof}

In particular, we have
\begin{equation*}
A^{(n)}_1(k+\tau)=\frac1{(k+\tau)^n}
\end{equation*}
for $k\ge0$.
Thus it follows that
\begin{align*}
\zeta(n,\tau)
=\sum_{k=0}^\infty \frac1{(k+\tau)^n}
&=\sum_{k=0}^\infty A^{(n)}_1(k+\tau)
\\
&=\sum_{k=0}^\infty \int\dotsb\int_D e^{-(k+\tau)w}\,dwdt_1\dotsb dt_{n-1}
\\
&=\int\dotsb\int_D \frac1{1-e^{-w}}e^{-\tau w}\,dwdt_1\dotsb dt_{n-1}.
\end{align*}
Using the expression
\begin{equation}\label{bernoulli}
\frac1{1-e^{-w}}
=\sum_{k=0}^\infty (-1)^k\frac{B_k}{k!}w^{k-1},
\end{equation}
we may continue the calculation above formally as follows:
\begin{align*}
\zeta(n,\tau)
&=\int\dotsb\int_D \biggl(\sum_{k=0}^\infty (-1)^k\frac{B_k}{k!}w^{k-1}\biggr)
e^{-\tau w}\,dwdt_1\dotsb dt_{n-1}
\\
&=\int\dotsb\int_D w^{-1}e^{-\tau w}\,dwdt_1\dotsb dt_{n-1}
+\int\dotsb\int_D \biggl(\sum_{k=1}^\infty (-1)^k\frac{B_k}{k!}w^{k-1}\biggr)
e^{-\tau w}\,dwdt_1\dotsb dt_{n-1}.
\end{align*}
By \eqref{change of variables} again, we have
\begin{align*}
\int\dotsb\int_D w^{-1}e^{-\tau w}\,dwdt_1\dotsb dt_{n-1}
&=\int\dotsb\int_\Omega r^{-1}e^{-\tau r}\,r^{n-1}\,drds_1\dotsb ds_{n-1}
\\
&=\vol(\Omega_0)\times\int_0^\infty r^{n-2}e^{-\tau r}\,dr
\\
&=\frac1{(n-1)\tau^{n-1}},
\end{align*}
that is \eqref{An} holds also for $j=0$.
We now have the following.

\begin{lem}\label{FPS}
There is a formal power series expression of $\zeta(n,\tau)$ as
\begin{equation}\label{FPSHZ}
\zeta(n,\tau)
=\frac1{(n-1)\tau^{n-1}}+\sum_{k=1}^\infty (-1)^k\frac{B_k}{k!}A^{(n)}_k(\tau)
=\sum_{k=0}^\infty (-1)^k\frac{B_k}{k!} \frac{(k+n-2)!}{(n-1)!}\frac1{\tau^{k+n-1}}. 
\end{equation}
In particular, when $\tau=1$, we have formally 
\begin{equation}\label{PSZ}
\zeta(n)=\frac1{n-1}+\frac1{n-1}\sum_{k=1}^\infty (-1)^kB_k\binom{k+n-2}k.
\end{equation}
\qed
\end{lem}

The series expressing the $\zeta(n, \tau)$ for $n\in \Z_{>1}$ above are obviously divergent
for any value of $\tau$. 
Nevertheless, the formal power series expression $\zeta(2, \tau)$ is convergent
for any $\pabs{\tau}>1$ \cite{HC2007}.
In the next section, we discuss the significance of this formal power series \eqref{FPSHZ}
from the perspectives of Borel sums and $p$-adic fields.
However, when $\tau=1$, there is currently no such way to interpret \eqref{PSZ}.

\bigskip

We have a similar formal power series
which expresses the Hurwitz-type spectral zeta function $Z_{\mathrm{QRM}+\tau}(t)$ of the QRM.
In fact, since
$Z_{\mathrm{QRM}+\tau}(t)=\sum_{j=0}^\infty e^{-t(\mu_j+\tau)}$
and for $n\geq2$
\[
\zeta_{\mathrm{QRM}}(n,\tau)
=\sum_{j=0}^\infty 
\int_0^\infty dt_{n-1} \int_{t_{n-1}}^\infty dt_{n-2}\,\dotsb\, \int_{t_1}^\infty dt\,e^{-t(\mu_j+\tau)},
\]
the derivation is obtained in the same way as above{\color{red}:}
\begin{equation}\label{FPQRM}
\zeta_{\mathrm{QRM}}(n, \tau)
= \sum_{k=0}^\infty (-1)^k\frac{(RB)_k}{k!} \frac{(k+n-2)!}{(n-1)!}\frac1{\tau^{k+n-1}}. 
\end{equation}

\section{Two interpretations of the divergent series}
\label{sec:Interpretation}

In this section, we study from two points of view the divergent series expression for $\zeta(n,\tau)$
obtained in the Section \ref{sec:DivergentSeries}.
Concretely, we consider the Borel transform of the divergent series
and the relation with the $p$-adic Hurwitz zeta function
via the study of the convergence as $p$-adic series.

\subsection{Borel's sums}
\label{sec:Borel}

The aim of this section is to examine the Borel sums
for the formal power series \eqref{FPSHZ} of Section \ref{sec:DivergentSeries}.

First, we recall the notion of the Borel summability.
We say that the formal power series
$A(z)\deq \sum_{n=0}^\infty a_nz^n$
is Borel summable if the following conditions are satisfied. 
\begin{enumerate}
\item
$B_A(t)\deq \sum_{n=0}^\infty \frac{a_n}{n!}t^n$
converges in some circle $\abs{t}<\delta$,
\item
$B_A(t)$ has an analytic continuation to a neighborhood of the positive real axis,
and 
\item
the Laplace transform
$\tB_A(z)\deq \frac1z\int_0^\infty e^{-t/z}B_A(t)dt$
converges (not necessarily absolutely) for some non-zero $z$.
\end{enumerate}
The series $B_A(t)$ is called the \emph{Borel transform} of the series $A(z)$,
and $\tB_A(z)$ is the \emph{Borel sum} of $A(z)$. 
A well-known theorem of Watson gives a sufficient condition for the function $A(z)$
to equal the Borel sum $\tB_A(z)$ of its asymptotic Taylor series.
Moreover, an improvement (relaxing the condition) of Watson's theorem is known \cite{Sakal2008}. 

\begin{ex} \label{Borel1} 
Recall the formal power series expression \eqref{FPSHZ} of $\zeta(n,\tau)$. 
\[
\zeta(n,\tau)
=\frac1{n-1}\sum_{k=0}^\infty (-1)^k B_k \binom{k+n-2}k\frac1{\tau^{k+n-1}}. 
\]
Put $z\deq \tau^{-1}$ and consider the series 
\begin{equation}
A(z) \deq z^{-n+1} \zeta(n, z^{-1})
= \frac1{n-1}\sum_{k=0}^\infty (-1)^k B_k \binom{k+n-2}k z^k. 
\end{equation}
Since 
\begin{align*}
\frac{t}{1-e^{-t}}=\sum_{k=0}^\infty (-1)^kB_k\frac{t^k}{k!},
\end{align*}
differentiating $\frac{t^{n-1}}{1-e^{-t}}$ for $(n-2)$-times, we have
\begin{align*}
\frac1{(n-1)!}\frac{d^{n-2}}{dt^{n-2}}\Bigl(\frac{t^{n-1}}{1-e^{-t}}\Bigr)
&=\frac1{(n-1)!}\sum_{k=0}^\infty (-1)^k \frac{B_k}{k} (k+n-2)(k+n-3)\cdots(k+1)t^k
\\
&=\frac1{n-1}\sum_{k=0}^\infty (-1)^k \frac{B_k}{k!} \binom{k+n-2}{k} t^k,
\end{align*}
and this is equal to the Borel transform of $A(z)$, that is,
\[
B_A(t)= \frac1{(n-1)!}\frac{d^{n-2}}{dt^{n-2}}\Bigl(\frac{t^{n-1}}{1-e^{-t}}\Bigr).
\]
Hence the Borel sum $\tB_A(z)$ of the series $A(z)$ is calculated by integral by parts as
\begin{align*}
\tB_A(z)
&=\frac1{(n-1)!}\frac1z\int_0^\infty e^{-t/z}\frac{d^{n-2}}{dt^{n-2}}\Bigl(\frac{t^{n-1}}{1-e^{-t}}\Bigr)dt
\\
&=\frac1{\Gamma(n)}\frac1{z^{n-1}}\int_0^\infty e^{-t/z}\Bigl(\frac{t^{n-1}}{1-e^{-t}}\Bigr)dt
=\frac1{z^{n-1}}\zeta(n, z^{-1}). 
\end{align*}
This shows that the Borel sum of the formal power series \eqref{FPSHZ} representing $\zeta(n,\tau)$
actually corresponds to $\zeta(n,\tau)$, as expected.
\qed
\end{ex}

\begin{ex}\label{Borel2QRM}
The Borel sum of the formal power series \eqref{FPQRM} of $\zeta_{\mathrm{QRM}+\tau}(n)$
is given by $\zeta_{\mathrm{QRM}+\tau}(n)$.
In other words, through the Borel sum,
the special values $\zeta_{\mathrm{QRM}+\tau}(n)$ at a positive integer is given
by the infinite sum of the special values of $\zeta_{\mathrm{QRM}+\tau}(s)$
at non-positive integers $s=0, -1, -2, \ldots$ via Lemma \ref{SVatNegative}.
A similar result holds also for the AQRM.
\qed
\end{ex}

Although for a general $s\in\C$ there is no chance to make a manipulation in the previous section,
we next generalize the Example \ref{Borel1} to a complex number $s\in \C$.

Clearly, the formal computations in Section \ref{sec:DivergentSeries} do not hold for general $s\in\C$.
Nevertheless, it is possible to generalize Example \ref{Borel1} for a complex number $s\in \C$.
First, we consider the integral expression of the Hurwitz zeta function of $\zeta(s,\tau)$
and integrate term-by-term the Taylor series expansion as
\begin{align}\label{FPSHZ-integral}
\frac1{\Gamma(s)} \int_0^\infty \Bigl(\frac{t^{s-1}}{1-e^{-t}}\Bigr)e^{-t\tau}dt
&= \frac1{\Gamma(s)} \sum_{k=0}^\infty (-1)^k\frac{B_k}{k!} \int_0^\infty t^{k+s-2}e^{-t\tau}dt
\\
&=\sum_{k=0}^\infty (-1)^k B_k \frac{\Gamma(k+s-1)}{\Gamma(k+1)\Gamma(s)} \frac1{\tau^{k+s-1}},
\end{align}
giving a formal power series expression of $\zeta(s,\tau)$ (eq. \eqref{FPSHZ}).

Let $z \deq \tau^{-1}$ and define the formal power series $A_s(z)$ by
\begin{equation}\label{FPSHZ(s)}
A_s(z) \deq z^{-s+1} \zeta(s, z^{-1})
=\sum_{k=0}^\infty (-1)^k B_k \frac{\Gamma(k+s-1)}{\Gamma(k+1)\Gamma(s)} z^k. 
\end{equation}

We now compute the Borel sum of this series
using a fractional integral/derivative (Euler's type integral transform) defined by 
\begin{align}\label{FractionalInt}
(J^\alpha f)(x) \deq \frac1{\Gamma(\alpha)} \int_0^x (x-t)^{\alpha-1} f(t)dt
\end{align}
for suitable function $f(t)$. 
Then, following Example \ref{Borel1}, when $s=n\in \Z_{>1}$, we obtain the following result. 
\begin{prop}\label{Borel2} 
The Borel sum of $A_s(z)$ in \eqref{FPSHZ(s)} exists and given by the (equivalent) formulas
\begin{equation}
\tB_{A_s}(z)=\frac{\sin\pi s}{z \pi (1-s)} \int_0^1 (1-x)^{-s+1}x^{s-3}\zeta(2,\tfrac1{xz})dx
=\frac1{\Gamma(s)}J^{2-s}(z^{s-3}\zeta(2,\tfrac1z)).
\end{equation}
\end{prop}

\begin{proof}
Looking at each coefficient of the expansion
\[
\frac{t^{s-1}}{1-e^{-t}}= \sum_{k=0}^\infty(-1)^k\frac{B_k}{k!}t^{k+s-2},
\]
we observe 
\begin{align*}
J^\alpha (t^{k+s-2})
&= \frac1{\Gamma(\alpha)} \int_0^t (t-x)^{\alpha-1} x^{k+s-2} dx
= \frac{t^{\alpha+k+s-2}}{\Gamma(\alpha)}\int_0^1 (1-y)^{\alpha-1}y^{k+s-2} dy 
\\
&= \frac{t^{\alpha+k+s-2}}{\Gamma(\alpha)}B(\alpha, k+s-1)
= \frac{\Gamma(k+s-1)}{\Gamma(\alpha+k+s-1)}t^{\alpha+k+s-2}.
\end{align*}
where we used the change of variables $x \to y$, $x=ty$. 
Hence, if we put $\alpha=-s+2$, then
\begin{align}\label{FractHrwitz}
\frac1{\Gamma(s)}J^{-s+2}\Bigl(\frac{t^{s-1}}{1-e^{-t}}\Bigr)
= \sum_{k=0}^\infty (-1)^k\frac{B_k}{k!} \frac{\Gamma(k+s-1)}{\Gamma(k+1)\Gamma(s)}t^k.
\end{align}
By the Stirling approximation
$\Gamma(z)= \sqrt{\frac{2\pi}{z}}\bigl(\frac{z}{e}\bigr)^z\bigl(1+O(\frac1z)\bigr)$
($\abs{\arg z}< \pi - \epsilon$, $\epsilon$ being positive)
and the Weierstrass M-test,
the series converges absolutely and uniformly in a small circle $\abs{t}<\sigma$
and an analytic continuation to a neighborhood of the positive real axis. 
Hence the series \eqref{FractHrwitz} gives the Borel transform of $A_s(z)$, denoted $B_{A_s}(t)$.

We next compute the Borel sum of $A_s(z)$ using the left hand side \eqref{FractHrwitz}.
Then, 
\begin{align*}
B_{A_s}(t)
&= \frac1{\Gamma(s)}J^{-s+2}\Bigl(\frac{t^{s-1}}{1-e^{-t}}\Bigr)
= \frac1{\Gamma(s)\Gamma(2-s)}\int_0^t(t-y)^{-s+1}\frac{y^{s-1}}{1-e^{-y}}dy
\\
&= \frac{\sin\pi s}{\pi (1-s)} t \int_0^1\frac{(1-x)^{-s+1}x^{s-1}}{1-e^{-tx}}dx,
\end{align*}
where the manipulation is legitimate by $1<\Re(s)<2$.

It follows that the Borel sum $\tB_{A_s}(z)$ of $A_s(z)$ is 
\begin{align*}
\tB_{A_s}(z)
=\frac{\sin\pi s}{z \pi (1-s)} \int_0^\infty e^{-t/z}
\biggl\{ t\int_0^1 \frac{(1-x)^{-s+1}x^{s-1}}{1-e^{-tx}}dx\biggr\} dt
\end{align*}
Changing the order of the integrals, we have 
\begin{align*}
\tB_{A_s}(z)
&= \frac{\sin\pi s}{z \pi (1-s)} \int_0^1 (1-x)^{-s+1}x^{s-1}
\biggl\{ \int_0^\infty \frac{te^{-t/z}}{1-e^{-tx}}dt \biggr\} dx
\\
&= \frac{\sin\pi s}{z \pi (1-s)} \int_0^1 (1-x)^{-s+1}x^{s-3}\zeta\Bigl(2,\frac1{xz}\Bigr)dx. 
\end{align*}
Since, it is clear that for $t>0$
\[
\sum_{n=0}^\infty (1+nt)^{-2} \simeq \int_0^\infty \frac{dy}{(1+yt)^2} =\frac1t
\quad(\text{asymptotically equal}),
\]
we observe
\begin{align*}
\tB_{A_s}(z)
&= \frac{\sin\pi s}{z \pi (1-s)} z^2 \int_0^1 (1-x)^{-s+1}x^{s-1}\sum_{n=0}^\infty (1+nxz)^{-2} dx
\\
&\simeq \frac{\sin\pi s}{z \pi (1-s)} z^2 \int_0^1 (1-x)^{-s+1}x^{s-1} \frac1{xz}dx
\\
&= \frac{\sin\pi s}{\pi (1-s)} \int_0^1 (1-x)^{-s+1}x^{s-2}
= \frac{\sin\pi s}{\pi (1-s)} B(s-1, 2-s)
= \frac1{s-1}. 
\end{align*}
Since we are assuming $1<\Re(s)<2$,
the Borel sum $\tB_{A_s}(z)$ of the formal power series $A_s(z)$ in \eqref{FPSHZ(s)} exists.
The equivalence of two expressions follows immediately from the definition of $J^{2-s}$
\end{proof}

\begin{rem}
From the viewpoint of the classical relation
(also \eqref{cor:DD-relation} for the spectral zeta function $\zeta_{\mathrm{QRM}}(s,\tau)$)
\begin{align*}
\frac{\partial^n}{\partial \tau^n} \zeta(s,\tau)=(-1)^n(s)_n \zeta(s+n,\tau), 
\end{align*}
the second expression of $\tB_{A_s}(z)$ in Proposition \ref{Borel2}, 
we may interpret the value $\tB_{A_s}(z)$ as a value, or shift of value, of $\zeta(s,\tfrac1z)$.
\qed
\end{rem}

\begin{rem}
Using Euler's integral formula of the Gaussian hypergeometric function
\[
\hgf21{a,b}{c}{z}
=\frac{\Gamma(c)}{\Gamma(a)\Gamma(c-a)}\int_0^1 t^{a-1}(1-t)^{c-a-1}(1-tz)^{-b}dt
\qquad (0< \Re(a)<\Re(c),\; \abs{z}<1),
\]
we have a formal expression of $\tB_{A_s}(z)$ $(0<\Re(s)<2)$ as 
\begin{align*}
\tB_{A_s}(z)= z\sum_{n=0}^\infty \hgf21{s,2}{2}{-nz}
&= \frac{z}{\Gamma(s)}\sum_{n=0}^\infty\sum_{m=0}^\infty (-1)^m\frac{\Gamma(s+m)}{m!}n^mz^m\\
&= \frac1{\Gamma(s)}\sum_{m=0}^\infty (-1)^m\frac{\Gamma(s+m)}{m!}\zeta(-m)z^{m+1}.
\end{align*}
\qed
\end{rem} 

\begin{ex}
We have the following formal power series expression of $\zeta(s,\tau+x)$
involving the Bernoulli polynomials by replacing $e^{-t\tau}$
by $e^{-t(x+\tau)}=e^{-tx}e^{-t\tau}$ and using the generating function for $B_k(x)$ at the integral. 
\begin{equation}\label{FPSHZ(s)2}
\zeta(s, \tau+x)
=\tau^{1-s}\sum_{k=0}^\infty (-1)^k B_k(x) \frac{\Gamma(k+s-1)}{\Gamma(k+1)\Gamma(s)} \tau^{-k}. 
\end{equation}
We will find this as a $p$-adically convergent series in the following section. 
\qed
\end{ex}

\subsection{$p$-adic Hurwitz zeta function}
\label{sec:p-adic-Hurwitz}

In this short section we observed that the divergent series, or formal power series,
that appeared in the previous section are convergent as $p$-adic series for sufficiently large $\pabs{\tau}$,
where $\pabs{~}$, is $p$-adic norm for any prime $p$.
Actually, certain cases of these formal power series expressing $\zeta(n,\tau)$
have been identified by the $p$-adic Hurwitz zeta function in the book \cite{HC2007}.

For simplicity, in this section we assume $p$ is an odd prime.
As usual, we denote by $\Qp$ (resp. $\Zp$) the field of $p$-adic numbers
(resp. the ring of $p$-adic integers), i.e., $\Qp$ is the fraction field of $\Zp$. 
Moreover, let us denote by $\Cp$ the completion of the algebraic closure
$\cQp$ of $\Qp$ and put $\CZp\coloneqq \Qp\setminus\Zp$.
Let $\omega$ be the Teichmüller character at a prime $p$.
Namely, $\omega$ is a character of $(\Z/p\Z)^\times$
taking values in the roots of unity of $\Zp$ (see \cite{HC2007} for the details).
Recall now the $p$-adic integral (Volkenborn's integral) defined by 
\[
\int_{\Zp} f(t)dt \coloneqq \lim_{r\to\infty}\frac1{p^r}\sum_{k=0}^{p^r-1} f(n)
\]
for a function $f\colon\Zp\to\Cp$ if it exists.
Since $\Z_{>0}$ is dense in $\Zp$, motivated by the formula
\[
-\frac1k \int_{\Zp}(x+\tau)^kdt
= -\frac{B_k(\tau)}{k}
= \zeta(1-k, \tau)
\]
for $k\in \Z_{>0}$ and $x\in\Q$,
Cohen introduced in \cite{HC2007} the $p$-adic Hurwitz zeta function $\zeta_p(s, \tau)$ as follows. 
\begin{dfn} \label{p-zeta}
For $s\in \Cp\setminus\{1\}$ such that $\pabs{s}<p^{\frac{p-2}{p-1}}$ and $\tau\in \CZp$,
the $p$-adic Hurwitz zeta function $\zeta_p(s, \tau)$ is defined by the (equivalent) formulas 
\begin{align}
\zeta_p(s, \tau)
\deq \frac1{s-1}\int_{\Zp} \langle \tau+t \rangle^{1-s} dt 
= \frac{\langle \tau \rangle^{1-s}}{s-1} \sum_{k=0}^\infty\binom{1-s}{k}B_j\tau^{-k},
\end{align}
where $\langle \tau \rangle= \tau/\omega(\tau)$. 
\end{dfn}

By the reflection formula of $\Gamma(z)$ notice that
\[
\frac1{s-1} \binom{1-s}{k}
= \frac1{s-1}\frac{\Gamma(2-s)}{\Gamma(k+1)\Gamma(2-s-k)}
= \frac{\Gamma(1-s)}{\Gamma(k+1)\Gamma(2-s-k)}
=(-1)^k \frac{\Gamma(k+s-1)}{\Gamma(k+1)\Gamma(s)}.
\]
This shows that the series \eqref{FPSHZ-integral} is convergent $p$-adically,
in other words, that gives $\zeta_p(s, \tau)$ if we think it as a $p$-adic series.
Moreover, since 
\[
B_n(x)=\sum_{j=0}^n \binom{n}{j} B_j x^{n-j}
\]
and $\langle \tau+x \rangle = (1+x/\tau) \langle \tau\rangle$,
if $\tau/x\in \CZp$, 
we have the following $p$-adic Laurent series (Proposition 11.2.9 (2) in \cite{HC2007})
\begin{align}
\zeta_p(s, \tau+x)
= \frac{\langle \tau \rangle^{1-s}}{s-1} \sum_{k=0}^\infty\binom{1-s}{k}B_k(x)\tau^{-k}. 
\end{align}
This is the same form of the series in \eqref{FPSHZ(s)2}
if we replace $\tau^{1-s}$ by $\langle \tau \rangle^{1-s}$. 

\begin{rem}\label{p-adicHO}
We make two remarks which are related to the present study;
there is a similar situation about the archimedean divergent but non-archimedean convergent series,
which look quite similar in the different context,
the second is concerning certain representation theory of $p$-adic group $SL_2(\Qp)$
behind the $p$-adic harmonic oscillator and its associated spectral zeta function
(cf. Figure \ref{fig.diveregent}). 

\begin{enumerate}
\item
It is known that the Riemann zeta function $\zeta(s)$ is not a solution
of any algebraic ordinary differential equations.
In 2015, Van Gorder \cite{VG2015} proved that
$\zeta(s)$ satisfies a non-algebraic differential equation
and showed that it formally satisfies an infinite order linear differential equation.
Moreover, Prado and Klinger-Logan \cite{PK2020} recently extended the result
to show that the Hurwitz zeta function $\zeta(s, \tau)$ also formally satisfies
a similar differential equation.
They proved, however, that the infinite order operator $T$
applied to the Hurwitz zeta function $\zeta(s, \tau)$ does not converge
at any point in the complex plane $\C$.
In this situation, recently Hu and Kim \cite{HK2021} shows that 
there is a $p$-adic analog of such infinite order operator $T_p$
such that the $p$-adic analog of the Hurwitz-type Euler zeta function $\zeta_{p,E}(s,\tau)$
(called the $p$-adic Hurwitz-type Euler zeta function,
i.e. $p$-adic analog of the Hurwitz-type Dirichlet $L$-function)
has a nice property.
Actually, in contrast with the complex case, due to the non-archimedean property,
the operator $T_p$ applied to $\zeta_{p,E}(s,\tau)$ is convergent $p$-adically
in the area of $s\in \Zp$ with $s\ne1$ and $\tau\in\Qp$ with $\pabs{\tau} > 1$.
The divergent series of the present paper resemble this fact,
so obtaining a reasonable understanding of both situations and their possible relation
may be an important subject of study.

\item 
The standard quantum mechanics starts with a representation
of the well-known Heisenberg canonical commutation relation in the space $L^2(\R)$.
In the Schrödinger representation the coordinates and momentum are realized by multiplication $x$
and differentiation $\frac{d}{dx}$, respectively.
However, in the $p$-adic quantum mechanics
we consider that $x \in \Qp$ and the wave function $\phi(x)$ takes value in $\C$.
Therefore, in $p$-adic quantum mechanics (see e.g. \cite{VV1989}),
we consider the representation theory of $p$-adic group (see, e.g. \cite{HCvG1970}).
Namely, in contrast with the standard way in the real case \cite{HT1992},
consider the Weyl representation in $L^2(\Qp)$ \cite{Weyl1931},
which is closely related to the Segal-Shale-Weil (or oscillator) representation \cite{Weil1964}
of the $p$-adic symplectic group $Sp_2(\Qp)$($\simeq SL_2(\Qp)$).
Then, natural questions arise;
Can one obtain the Hurwitz-type spectral zeta function for the $p$-adic harmonic oscillator?
Moreover, the $p$-adic Hurwitz zeta function does give such a Hurwitz-type spectral zeta function? 

It is also remarked that the partition function
of the $p$-adic (and adelic) harmonic oscillator is calculated recently in \cite{RN2022}.
Here, the $p$-adic propagator (heat kernel) is suitably calculated by using Feynman's path integral method.
Actually, the form of the partition resembles to the one for the standard qHO
(by normalizing the Planck constant $\hbar=1$ and angular frequency $\omega=1$) as
\[
Z_p(\beta)=\pabs{4\sinh^2\Bigl(\frac{\beta}2\Bigr)}^{-\frac12}. 
\]
Bearing in mind with $\pabs{\sinh(x)}=\pabs{x}$, we have $Z_p(\beta)=\pabs{\beta}^{-1}$.
It would be interesting to get any interpretation that
the $p$-adic Hurwitz zeta function $\zeta_p(s, \tau)$ can be obtained
from $Z_p(\beta)$ with a rational process. 

\end{enumerate}
\end{rem}

\begin{rem}
In view of Remark \ref{RBP}, when $g, \Delta\in \Q$, we may naturally think of a $p$-adic analog of the spectral zeta function for the QRM. 
\end{rem}

\section{Concluding remarks} 
\label{sec:SV_NCHO}

If Conjecture \ref{QPF_NCHO} is true,
computing the $n$-th integral of the partition function
$Z_Q(t,\tau)\deq\sum_{j=0}^\infty e^{-t(\lambda_j+\tau)}$ 
for $\tau>0$ for the NCHO by the same way as in Section \ref{sec:DivergentSeries},
we have the following formal power series expression of $\zeta_Q(n,\tau)$ for $n\geq 2$.
\begin{align}\label{FPSNCZeta}
\zeta_Q(n,\tau)
& \sim 2\sum_{k=0}^\infty (-1)^k\frac{B_{Q, k}}{k!} \frac{(k+n-2)!}{(n-1)!}\frac1{\tau^{k+n-1}}
\\
&= \sum_{k=1}^\infty (-1)^{k-1} \frac{\zeta_Q(1-k)}{(k-1)!} \frac{(k+n-2)!}{(n-1)!}\frac1{\tau^{k+n-1}}
+\frac1{n-1}\Res_{s=1}\zeta_Q(s)\frac1{\tau^{n-1}}
\\
& \eqqcolon A_{Q,n}(\tau^{-1}).
\end{align}
We can show in the same way as in Example \ref{Borel1} that
the Borel sum of this formal power series $A_{Q,n}(z)$ equals $\zeta_Q(n,\frac1z)$
(see also Example \ref{Borel2QRM} for the QRM).
Although we can not expect any functional equation for the spectral zeta function $\zeta_Q(s)$
(and $\zeta_Q(s,\tau)$) for the NCHO,
this shows some implication about the special values of $\zeta_Q(s,\tau)$
at the positive integer $s=n$ can be seen by the (sum of) special values at the negative integers
via the Borel sum $\tB_{A_{Q,n}}(\tau^{-1})$.
In addition, if we assume the parameters $\alpha, \beta \in \Q$,
we observe that this series converges $p$-adically.
However, there remains a problem about taking $\tau$ to be $1$
in the corresponding both expressions.
Therefore, we have the following 

\begin{prob} 
Study the special values of the Hurwitz-type spectral zeta function $\zeta_Q(s,\tau)$ as in \cite{IW2005KJM}. 
Notice that, at least formally, we have 
\begin{align*}
\zeta_Q(n,\tau)
&= \Tr(Q+\tau)^{-n}= \Tr Q^{-n}(1+\tau Q^{-1})^{-n}
\\
&= \Tr Q^{-n} \sum_{k=0}^\infty \binom{-n}{k}Q^{-k}\tau^k
= \sum_{k=0}^\infty (-1)^k\binom{n+k-1}{k} \Tr Q^{-n-k} \tau^k
\\
&= \sum_{k=0}^\infty (-1)^k\binom{n+k-1}{k}\zeta_Q(n+k) \tau^k,
\end{align*}
but the target is to get an explicit integral representation of the special value $\zeta_Q(n,\tau)$
as in Subsection \ref{sec:Geometry-NCHO-Zeta}. 
The problem is to understand the special values of positive integers
for $\zeta_Q(s,\tau)$ for the NCHO \cite{IW2005KJM}
in terms of modular forms \cite{KW2006KJM, O2008RJ} and Eichler forms \cite{KW2023}
(see, e.g., Eichler forms \cite{G1961}, also \cite{B1975Cr})
because of a lack of any functional equation. 
\qed
\end{prob}






\appendix
\section{Apéry and Apéry-like numbers from a comparative perspective}\label{Appendix_Apery}

\subsection{Recurrence relations and differential equations}\label{A1}

Both of the Apéry numbers and Apéry-like numbers
satisfy certain three-term recurrence relation,
and hence each of their generating functions satisfies an ordinary differential equation
correspondingly.
However, there are some differences between them.

For each special value $\zeta(2)$ and $\zeta(3)$,
there are two kinds of \emph{paired} Apéry numbers $\apery k(n)$ and $\bpery k(n)$ $(k=2,3)$.
Their ratio $\bpery k(n)/\apery k(n)$ is used to prove the irrationality of $\zeta(k)$.
The generating functions $\Apery2(t)$ and $\Apery3(t)$ for
\begin{equation*}
\apery2(n)=\sum_{k=0}^n\binom nk^{\!\!2}\!\binom{n+k}k,
\qquad
\apery3(n)=\sum_{k=0}^n\binom nk^{\!\!2}\!\binom{n+k}k^{\!\!2}
\end{equation*}
satisfy \emph{homogeneous} differential equations
$L_2\Apery2(t)=0$ and $L_3\Apery3(t)=0$.
On the other hand,
the generating functions $\Bpery k(t)$ for $B_k(n)$ satisfy \emph{inhomogeneous} equations
$L_2\Bpery2(t)=-5$ and $L_3\Bpery3(t)=5$.
Each of them share the same operator $L_2$ and $L_3$
with the equations for $\Apery2(t)$ and $\Apery3(t)$ respectively.

In contrast, in our NCHO case,
we have \emph{unpaired} Apéry-like numbers $\J kn$
for each special value $\zeta_Q(k)$ $(k=2,3,4,\dots)$.
We can solve the recurrence \eqref{eq:recurrence_of_Jkn} for $\J kn$
and have a closed formula for $\J kn$,
which are kinds of binomial sums similar to those for Apéry numbers.
For instance, we have
\begin{align*}
\J2n&=\zeta\Bigl(2,\frac12\Bigr)
\sum_{k=0}^n(-1)^k\binom{-\frac12}{k}^{\!\!2}\!\binom{n}{k},\\
\J3n&=-\frac12\sum_{k=0}^n(-1)^k\binom{-\frac12}{k}^{\!\!2}\!\binom{n}{k}
\sum_{0\le j<k}\frac{1}{(j+\frac12)^3}\binom{-\frac12}{j}^{\!\!-2}
+2\zeta\Bigl(3,\frac12\Bigr)\sum_{k=0}^n(-1)^k\binom{-\frac12}{k}^{\!\!2}\binom{n}{k},\\
\J4n&=-\zeta\Bigl(2,\frac12\Bigr)\sum_{k=0}^n(-1)^k\binom{-\frac12}{k}^{\!\!2}\!\binom{n}{k}
\sum_{0\le j<k}\frac{1}{(j+\frac12)^2}
+3\zeta\Bigl(4,\frac12\Bigr)\sum_{k=0}^n(-1)^k\binom{-\frac12}{k}^{\!\!2}\!\binom{n}{k}.
\end{align*}
The generating function $w_k(z)$ for $\J kn$ satisfies an differential equation
of the form $Dw_k(z)=w_{k-2}(z)$, which is homogeneous only if $k=2$ ($w_0(z)=0$).
Unlike the original Apéry number case, this is a `ladder' relation. Furthermore, these ladder relations are characterized by being separated into odd-numbered and even-numbered families (see 
\eqref{eq:recurrence_of_Jkn} and \eqref{Ladder structure}). 

\subsection{Congruence relations for Apéry(-like) numbers}

We collect some congruence relations for Apéry(-like) numbers
for reference and comparison and to observe the higher remainder terms
in \eqref{General-Int-Expression} for Apéry-like numbers.
See Table \ref{AperyComparison} below.

It is well-known that Apéry numbers $\apery2(n)$ and $\apery3(n)$ have various kinds of congruence properties.
Here we pick up several of them (see \cite{Beu1985, SB1985, B1987})
and those of Apéry-like numbers \cite{KW2006KJM, KW2023}. 

\begin{prop}
Let $p$ be a prime and $n=n_0+n_1p+\dots+n_kp^k$ be the $p$-ary expansion of $n\in\Z_{\ge0}$ $(0\le n_j<p)$.
Then it holds that
\begin{align}
\apery2(n)&\equiv\prod_{j=0}^k\apery2(n_j) \pmod p,
\\
\apery3(n)&\equiv\prod_{j=0}^k\apery3(n_j) \pmod p.
\end{align}
\end{prop}

\begin{prop}\label{prop:Apery_congruence}
For all odd prime $p$, it holds that
\begin{align}
\apery2(mp^r-1)&\equiv\apery2(mp^{r-1}-1)\pmod{p^r},
\\
\apery3(mp^r-1)&\equiv\apery3(mp^{r-1}-1)\pmod{p^r}
\end{align}
for any $m,r\in\Z_{>0}$.
These congruence relations hold modulo $p^{3r}$ if $p\ge5$
{\upshape(}known and referred to as \emph{supercongruences}{\upshape)}.
\end{prop}


\begin{prop}
For any odd prime $p$ and any $m,r\in\Z_{>0}$ with $m$ odd, it holds that
\begin{align}
\apery2(\tfrac{mp^r-1}2)-\lambda_p\apery2(\tfrac{mp^{r-1}-1}2)
+(-1)^{(p-1)/2}p^2\apery2(\tfrac{mp^{r-2}-1}2)\equiv0 \pmod{p^r},\label{eq:ASD-type_congruence_for_A2}
\\
\apery3(\tfrac{mp^r-1}2)-\gamma_p\apery3(\tfrac{mp^{r-1}-1}2)
+p^3\apery3(\tfrac{mp^{r-2}-1}2)\equiv0 \pmod{p^r}.\label{eq:ASD-type_congruence_for_A3}
\end{align}
Here $\lambda_n$ and $\gamma_n$ are defined by
\begin{equation*}
\sum_{n=1}^\infty \lambda_nq^n=\eta(4\tau)^6,\qquad
\sum_{n=1}^\infty \gamma_nq^n=\eta(2\tau)^4\eta(4\tau)^4.
\end{equation*}
\end{prop}

We give several congruence relations among \emph{normalized} Apéry-like numbers
$\tJ2n, \tJ3n\in\Q$ which are defined to be
\begin{equation*}
\J2n=\J20\tJ2n,\quad
\J3n=\tJ3n+\J30\tJ2n.
\end{equation*}
See \cite{KW2006KJM}.

\begin{prop}
Let $p$ be a prime and $n=n_0+n_1p+\dots+n_kp^k$ be the $p$-ary expansion of $n\in\Z_{\ge0}$ $(0\le n_j<p)$.
Then it holds that
\begin{equation}
\tJ2n\equiv\prod_{j=0}^k\tJ2{n_j} \pmod p.
\end{equation}
\end{prop}

\begin{prop}
For any odd prime $p$ and positive integers $m,r$, the congruence relation
\begin{align}
\tJ2{mp^r}&\equiv\tJ2{mp^{r-1}}\pmod{p^r},
\\
\tJ3{p^r}p^{3r}&\equiv\tJ3{p^{r-1}}p^{3(r-1)}\pmod{p^r}.
\end{align}
holds.
\end{prop}

\begin{thm}[Long-Osburn-Swisher \cite{LOS2014}]
For any odd prime $p$, the congruence relation
\begin{equation}
\sum_{n=0}^{p-1} \tJ2n^2 \equiv (-1)^{\frac{p-1}2} \pmod{p^3}
\end{equation}
holds.
\end{thm}
Another generalization of supercongruences can been seen in \cite{L2018}. 

\begin{table}[htb]
\centering
\begin{tabular}{c|c}
Apéry numbers
& congruence with respect to $p$-ary expansion \\
& congruence between $(mp^r-1)$-th and $(mp^{r-1}-1)$-th numbers \\
& congruence among three terms \\
\hline
Apéry-like numbers
& congruence with respect to $p$-ary expansion \\
& congruence between $mp^r$-th and $mp^{r-1}$-th numbers \\
& congruence for square-sum
\end{tabular}
\caption{congruence relations}\label{AperyComparison}
\end{table}

In general, the \emph{normalized Apéry-like numbers} $\tJ{k}n$ for $k=2,3,4,\ldots$
are defined inductively by the conditions
\begin{align*}
\J{2s}n &= \sum_{j=0}^{s-1} \J{2s-2j}0\tJ{2j+2}n,
\\
\J{2s+1}n &= \sum_{j=0}^{s-1} \J{2s+1-2j}0\tJ{2j+2}n+\tJ{2s+1}n.
\end{align*}
It is equivalent to define $\tJ kn$ by the recurrence relation
\begin{equation*}
\tJ kn=\J{r}0^{-1}\Bigl(\J kn-\sum_{0<2j<k} \J{k-2j}0\tJ{2j+2}n\Bigr),\qquad
r=\begin{cases}
1 & k\equiv1\pmod2, \\[.2em]
2 & k\equiv0\pmod2.
\end{cases}
\end{equation*}
These are \emph{rational numbers} and satisfy the relation
\begin{equation}\label{eq:reccurence-for-tJ}
4n^2\tJ kn-(8n^2-8n+3)\tJ k{n-1}+4(n-1)^2\tJ k{n-2}=4\tJ{k-2}{n-1},
\end{equation}
which is identical to the one for $J_k(n)$.
$\tJ kn$ has an explicit expression as follows.

\begin{thm}[{\cite[Theorem 3.5]{KW2023}}]
For $s=1,2,3,\dots$, we have
\begin{align*}
\tJ{2s+2}n&=\sum_{k=0}^n(-1)^k\binom{-\frac12}{k}^{\!\!2}\binom{n}{k}\evenZ_s(k),
\\
\tJ{2s+1}n&=\sum_{k=0}^n(-1)^k\binom{-\frac12}{k}^{\!\!2}\binom{n}{k}\oddZ_s(k),
\end{align*}
where
\begin{align*}
\evenZ_s(k)&\coloneqq (-1)^s\sum_{k>j_1>\dots>j_s\ge0}\frac1{(j_1+\frac12)^2\dots(j_s+\frac12)^2},
\\
\oddZ_s(k)&\coloneqq \frac{(-1)^s}2\sum_{k>j_1>\dots>j_s\ge0}
\frac1{(j_1+\frac12)^2\dots(j_{s-1}+\frac12)^2(j_s+\frac12)^3}\binom{-\frac12}{j_s}^{\!\!-2}.
\end{align*}
\end{thm}

Since the numbers $\tJ kn$ are rational, we can discuss the congruent properties for them.
For instance, we have the following result.

\begin{thm}[{\cite[Theorem 3.10]{KW2023}}]\label{weaker version}
If $1\le m<\frac p2$, then
\begin{equation}
p^{2sn}\tJ{2s+2}{mp^n}\equiv p^{2s(n-1)}\tJ{2s+2}{mp^{n-1}} \pmod{p^n}
\end{equation}
holds.
\end{thm}

\subsection{Higher remainder terms in the special value formula}

We have defined the Apéry-like numbers $\J kn$ for every $k\ge2$
via the term $R_{k,1}(\kappa)$ in the special value formula for $\zeta_Q(k)$.
It is still a difficult task to define similar numbers for $R_{k,r}(\kappa)$ ($r\ge2$) in general.
Here we give the example of $R_{4,2}(\kappa)$. 
Let us put
\begin{equation*}
a(u)=1-u_1^2u_2^2u_3^2u_4^2,\qquad
b(u)=(1-u_1^4u_2^4)(1-u_3^4u_4^4),\qquad
c(u)=(1-u_1^4)(1-u_2^4)(1-u_3^4)(1-u_4^4)
\end{equation*}
for convenience. Then we have
\begin{align*}
R_{4,2}(t)
&=16\int_{[0,1]^4}
\frac1{a(u)}
\frac{du_1du_2du_3du_4}{\sqrt{1+tb(u)/a(u)^2+(t+t^2)c(u)/a(u)^2}}
\\
&=16\sum_{n=0}^\infty\binom{-\frac12}n\int_{[0,1]^4}
\frac1{a(u)}\Bigl(t\frac{b(u)+c(u)}{a(u)^2}+t^2\frac{c(u)}{a(u)^2}\Bigr)^n\,du_1du_2du_3du_4
\\
&=16\sum_{n=0}^\infty\binom{-\frac12}n\sum_{k=0}^n\binom nk t^{n+k}
\int_{[0,1]^4} \frac{(b(u)+c(u))^{n-k}c(u)^k}{a(u)^{2n+1}}\,du_1du_2du_3du_4,
\end{align*}
where we set $t=\kappa^2$ and regard $R_{4,2}$ as a function in $t$.
Now we define
\begin{equation*}
A_{n,k}\coloneqq \int_{[0,1]^4} \frac{(b(u)+c(u))^{n-k}c(u)^k}{a(u)^{2n+1}}\,du_1du_2du_3du_4.
\end{equation*}
It follows that
\begin{equation*}
R_{4,2}(t)=16\sum_{n=0}^\infty\binom{-\frac12}n\sum_{k=0}^n\binom nk A_{n,k} t^{n+k}.
\end{equation*}
We may regard $A_{n,k}$ as an analog of Apéry-like numbers in this case.
We write
\begin{align*}
A_{n,k}
&=\sum_{j=0}^{n-k}\binom{n-k}j
\int_{[0,1]^4} \frac{b(u)^{n-k-j}c(u)^{k+j}}{a(u)^{2n+1}}\,du_1du_2du_3du_4
\\
&=\sum_{j=k}^n \binom{n-k}{j-k}
\int_{[0,1]^4} \frac{b(u)^{n-j}c(u)^{j}}{a(u)^{2n+1}}\,du_1du_2du_3du_4.
\end{align*}
If we put
\begin{align*}
B_{n,j}
&\coloneqq \int_{[0,1]^4} \frac{b(u)^{n-j}c(u)^j}{a(u)^{2n+1}}\,du_1du_2du_3du_4
\\
&=\int_{[0,1]^4} \frac{(1-u_1^4u_2^4)^{n-j}(1-u_3^4u_4^4)^{n-j}
(1-u_1^4)^j(1-u_2^4)^j(1-u_3^4)^j(1-u_4^4)^j}{(1-u_1^2u_2^2u_3^2u_4^2)^{2n+1}}\,du_1du_2du_3du_4,
\end{align*}
then we have
\begin{equation*}
A_{n,k}
=\sum_{j=k}^n\binom{n-k}{j-k} B_{n,j}
\end{equation*}
and
\begin{align*}
R_{4,2}(t)
&=16\sum_{n=0}^\infty\binom{-\frac12}n\sum_{k=0}^n\binom nk \sum_{j=k}^n\binom{n-k}{j-k} B_{n,j} t^{n+k}
\\
&=16\sum_{0\le k\le j\le n}\binom{-\frac12}n\binom nk\binom{n-k}{j-k} B_{n,j} t^{n+k}
\\
&=16\sum_{0\le k\le j\le n}\binom{-\frac12}n\binom nj\binom jk B_{n,j} t^{n+k}
\\
&=16\sum_{k,j,n\ge0}\binom{-\frac12}n\binom nj\binom jk B_{n,j} t^{n+k}.
\end{align*}
The numbers $A_{n,k}$ and $B_{n,j}$ are expected to be $\Q$-linear combinations
of Riemann zeta values $\zeta(2)=\frac{\pi^2}6$ and $\zeta(4)=\frac{\pi^4}{90}$.
Here we give a few examples on them:
\begin{align*}
A_{0,0}=\frac{1}{2^5\,3}\pi^4, \quad
A_{1,0}=\frac{1}{2^6}\pi^4-\frac{1}{2^6}\pi^2, \quad
A_{1,1}=\frac{1}{2^7}\pi^4-\frac{3^2}{2^8}\pi^2,
\\
B_{0,0}=\frac{1}{2^5\,3}\pi^4, \quad
B_{1,0}=\frac{1}{2^7}\pi^4+\frac{5}{2^8}\pi^2, \quad
B_{1,1}=\frac{1}{2^7}\pi^4-\frac{3^2}{2^8}\pi^2.
\end{align*}
As with the case of Apéry-like numbers,
it is hoped that recurrence relations between the numbers $A_{n,k}$ and $B_{n,j}$,
differential equations satisfied by their generating functions,
and their geometric interpretations can be obtained.
We leave these topics for future investigation.

\section*{Acknowledgements}

The authors would like to thank anonymous referees for the valuable comments. This also made the authors aware of the paper \cite{BZ2019JAMS}.
The authors would like to express their gratitude to Cid Reyes-Bustos
for giving them a number of valuable comments.


\bigskip

\begin{flushleft}
Kazufumi Kimoto \par
Department of Mathematical Sciences, \par
University of the Ryukyus \par
1 Senbaru, Nishihara, Okinawa 903-0213 JAPAN \par
\texttt{kimoto@sci.u-ryukyu.ac.jp} \par

\bigskip

Masato Wakayama \par
NTT Institute of Fundamental Mathematics, \par 
Nippon Telegraph and Telephone Corporation,\par 
3-9-11 Midori-cho Musashino-shi, Tokyo, 180-8585 JAPAN \par\par 
\texttt{masato.wakayama@ntt.com} 

\medskip

\& \par
Institute of Mathematics for Industry,\par
Kyushu University \par
744 Motooka, Nishi-ku, Fukuoka 819-0395 JAPAN \par
\texttt{wakayama@imi.kyushu-u.ac.jp}

\end{flushleft}

\end{document}